\documentclass[reqno, 12pt]{amsart}

\usepackage{array}
\usepackage{amsmath}
\usepackage{amsfonts}
\usepackage{amssymb}
\usepackage{enumerate}
\usepackage{amsthm}
\usepackage{amsmath, amscd}
\usepackage{bm}
\usepackage{xy}
\xyoption{all}
\usepackage{color}
\usepackage{marvosym}
\usepackage{amsbsy}
\usepackage{hyperref}
\usepackage{tikz}
\usetikzlibrary{matrix,arrows,backgrounds}
	
\newtheorem{introtheorem}{Theorem}

\newtheorem{introconjecture}{Conjecture}
\newtheorem{theorem}{Theorem}[subsection]

\newtheorem{lemma}[theorem]{Lemma}
\newtheorem{proposition}[theorem]{Proposition}
\newtheorem{corollary}[theorem]{Corollary}

\newtheorem{conjecture}[theorem]{Conjecture}

\theoremstyle{definition}
\newtheorem{definition}[theorem]{Definition}
\newtheorem{remark}[theorem]{Remark}

\newcommand{\op}[1]{\operatorname{#1}}

\newcommand{\leftexp}[2]{{\vphantom{#2}}^{#1}{#2}}

\newcommand{\modmod}[1]{/ \! \! / \!_{#1}}

\newcommand{\dbcoh}[1]{\operatorname{D}^{\operatorname{b}}(\operatorname{coh }#1)}

\newcommand{\weezer}{\leftexp{=}{\kern-0.23em\operatorname{W}}^{\kern-0.21em =}}

\newcommand{\sidenote}[1]{}

\newcommand{\fs}[1]{\op{Fuk}^\rightharpoonup ({#1})}
\newcommand{\langinz}{X^{LG}}

\def\ker{\op{\mbox{Ker}}}

\def\Z{\op{\mathbb{Z}}}
\def\C{\op{\mathbb{C}}}
\def\R{\op{\mathbb{R}}}

\def\O{\op{\mathcal{O}}}

\def\P{\op{\mathbb{P}}}

\def\ra{\rightarrow}

\def\tand{\text{ and } }

\title[The Mori Program and Toric HMS]{The Mori Program and Non-Fano Toric Homological Mirror Symmetry}

\author[Ballard]{Matthew Ballard}
\address{
  \begin{tabular}{l}
   Matthew Ballard  \\ 
   \hspace{.1in} University of Wisconsin-Madison, Department of Mathematics, Madison, WI, USA \\
   \hspace{.1in} Universit\"at Wien, Fakult\"at f\"ur Mathematik,  Wien, \"Osterreich \\
   \hspace{.1in} Email: {\bf ballard@math.wisc.edu} \\
  \end{tabular}
}

\author[Diemer]{Colin Diemer}
\address{
  \begin{tabular}{l}
   Colin Diemer  \\ 
   \hspace{.1in} University of Miami, Department of Mathematics, Coral Gables, FL, USA \\
   \hspace{.1in} Universit\"at Wien, Fakult\"at f\"ur Mathematik,  Wien, \"Osterreich\\ 
   \hspace{.1in} Email: {\bf cdiemer@gmail.com} \\
  \end{tabular}
}

\author[Favero]{David Favero}
\address{
  \begin{tabular}{l}
   David Favero \\
   \hspace{.1in} Universit\"at Wien, Fakult\"at f\"ur Mathematik,  Wien, \"Osterreich \\
   \hspace{.1in} Email: {\bf favero@gmail.com} \\
  \end{tabular}
}

\author[Katzarkov]{Ludmil Katzarkov}
\address{
  \begin{tabular}{l}
   Ludmil Katzarkov \\
   \hspace{.1in} University of Miami, Department of Mathematics, Coral Gables, FL, USA \\ 
   \hspace{.1in} Universit\"at Wien, Fakult\"at f\"ur Mathematik,  Wien, \"Osterreich \\
   \hspace{.1in} Email: {\bf lkatzark@math.uci.edu} \\
  \end{tabular}
}

\author[Kerr]{Gabriel Kerr}
\address{
  \begin{tabular}{l}
   Gabriel Kerr  \\ 
   \hspace{.1in} University of Miami, Department of Mathematics, Coral Gables, FL, USA \\ 
   \hspace{.1in} Email: {\bf gabriel.d.kerr@gmail.com} \\
  \end{tabular}
}

\numberwithin{equation}{section}
\begin{document}
\renewcommand{\labelenumi}{\emph{\alph{enumi})}}

\begin{abstract}
 In the case of toric varieties, we continue the pursuit of Kontsevich's fundamental insight, Homological Mirror Symmetry, by unifying it with the Mori program. We give a refined conjectural version of Homological Mirror Symmetry relating semi-orthogonal decompositions of the $B$-model on toric varieties to semi-orthogonal decompositions on the $A$-model on the mirror Landau-Ginzburg models. 
 
 As evidence, we prove a new case of Homological Mirror Symmetry for a toric surface whose anticanonical bundle is not nef, namely a certain blow-up of $\P^2$ at three infinitesimally near points.
 
\end{abstract}

\maketitle

\section{Introduction} \label{sec: Introduction}

In \cite{AKO1}, Auroux, Orlov, and the fourth named author discovered an ``unusual'' phenomenon regarding Homological Mirror Symmetry (HMS) for Hirzebruch surfaces.  This phenomenon also arose, in greater generality, in \cite{Abouzaid} where Abouzaid considered HMS for toric varieties whose anticanonical bundle is not ample.

The basic observation in the example of the Hirzebruch surface is the following. Consider the weighted projective stack, 
\[
\P(1:1:n) = [ (\mathbb{A}^3 \setminus \{0\}) / \mathbb{G}_m ]
\]
where $\mathbb{G}_m = \C^*$ acts with weights $1,1,n$. The toric fan has three rays with primitive generators $(1,0)$, $(0,1)$, and $(-1,-n)$. Following the Givental-Hori-Vafa prescription \cite{HoriVafa}, the mirror is a Landau-Ginzburg (LG) model
\[ 
w: (\C^*)^2 \to \C,
\]
where 
\[
w(x,y) = x+y+\frac{1}{x^n y}.
\]

Correspondingly, for the Hirzebruch surface 
\[
\mathbb{F}_n = \underline{\op{Proj}}(\O_{\P^1} \oplus \O_{\P^1}(n)),
\]
 the (naive) mirror is a Landau-Ginzburg model
\[
w': (\C^*)^2 \to \C,
\]
where 
\[
w'(x,y) = x+y+\frac{1}{x} + \frac{1}{x^n y}.
\]
It was observed in Lemma 5.3 of \cite{AKO1} that for $n \geq 2$ these two functions are isotopic as exact Lefschetz fibrations and therefore have equivalent Fukaya-Seidel categories.  

On the other hand, the bounded derived categories of coherent sheaves on $\P(1:1:n)$ and $\mathbb{F}_n$ are not equivalent for $n \not = 2$.  For example, a full strong exceptional collection for $\dbcoh{\mathbb{F}_n}$ is given by 
\[
T = \O_{\mathbb{F}_n} \oplus \pi^*\O_{\mathbb{P}^1}(1) \oplus \O_{\pi}(1) \oplus \pi^*\O_{\mathbb{P}^1}(1) \otimes \O_{\pi}(1)
\]
where $\pi: \mathbb{F}_n \to \P^1$ is the projection.
Similarly, a full strong exceptional collection for $\dbcoh{\P(1:1:n)}$ is given by
\[
T' = \bigoplus_{i=0}^{n+2} \O_{\mathbb{P}(1:1:n)}(i).
\]

\begin{figure}[ht]
\begin{center}
 \begin{tikzpicture}[ scale=1.2, level/.style={->,>=stealth,thick} ]
	\node (a) at (-2,2) {$\bullet$};
	\node (b) at (2,2) {$\bullet$};
	\node (c) at (2,-2) {$\bullet$};
	\node (d) at (-2,-2) {$\bullet$};
	\draw[level] (a) -- (d) node at (-2.2,0) {$x_2$};
	\draw[level] (b) -- (c) node at (2.25,0) {$x_2$};
	\draw[level] (-1.78,2.075) .. controls (-.9,2.3) and (.9,2.3)  .. (1.78,2.075) node at (0,2.4) {$x_0$};
	\draw[level] (a) -- (b) node at (0,1.85) {$x_1$};
	\draw[level] (-1.78,-2.075) .. controls (-.9,-2.3) and (.9,-2.3)  .. (1.78,-2.075) node at (0,-2.4) {$x_1$};
	\draw[level] (d) -- (c) node at (0,-1.85) {$x_0$};
	\draw[level] (1.925,1.775) .. controls (1.7,0) and (0,-1.7)  .. (-1.775,-1.925) node at (-.825,.825) {$x_0^3$};
	\draw[level] (1.775,1.925) .. controls (0,1.7) and (-1.7,0)  .. (-1.925,-1.775) node at (.8,-.8) {$x_1^3$};
	\draw[level] (1.855,1.785) .. controls (.5,0) and (0,-.5)  .. (-1.785,-1.855) node at (-.42,.42) {$x_0^2 x_1$};
	\draw[level] (1.785,1.855) .. controls (0,.5) and (-.5,0)  .. (-1.855,-1.785) node at (.4,-.4) {$x_0 x_1^2$};
 \end{tikzpicture}
\end{center}
\caption{Quiver representing the endomorphism algebra of $T$ $(n=4)$}
\label{fig: Hirzebruch Quiver}
\end{figure}

\begin{figure}[ht]

\begin{center}
\begin{tikzpicture}[
      scale=1.2,
      level/.style={->,>=stealth,thick}]
	\node (a) at (-5,0) {$\bullet$};
	\node (b) at (-3,0) {$\bullet$};
	\node (c) at (-1,0) {$\bullet$};
	\node (d) at (1,0) {$\bullet$};
	\node (e) at (3,0) {$\bullet$};
	\node (f) at (5,0) {$\bullet$};
	\draw[level] (-4.85,.1) .. controls (-4.4,.25) and (-3.6,.25)  ..  (-3.15,.1) node at (-4,.325) {$x_0$};
	\draw[level] (-4.85,-.1) .. controls (-4.4,-.25) and (-3.6,-.25)  .. (-3.15,-.1) node at (-4,-.35) {$x_1$};
	\draw[level] (-2.85,.1) .. controls (-2.4,.25) and (-1.6,.25)  .. (-1.15,.1) node at (-2,.325) {$x_0$};
	\draw[level] (-2.85,-.1)  .. controls (-2.4,-.25) and (-1.6,-.25)  .. (-1.15,-.1) node at (-2,-.35) {$x_1$};
	\draw[level] (-.85,.1) .. controls (-0.4,.25) and (.4,.25)  .. (.85,.1) node at (0,.325) {$x_0$};
	\draw[level] (-.85,-.1) .. controls (-0.4,-.25) and (.4,-.25)  ..  (.85,-.1) node at (0,-.35) {$x_1$};
	\draw[level] (1.15,.1) .. controls (1.6,.25) and (2.4,.25)  .. (2.85,.1) node at (2,.325) {$x_0$};
	\draw[level] (1.15,-.1) .. controls (1.6,-.25) and (2.4,-.25)  ..  (2.85,-.1) node at (2,-.35) {$x_1$};
	\draw[level] (3.15,.1) .. controls (3.6,.25) and (4.4,.25)  .. (4.85,.1) node at (4,.325) {$x_0$};
	\draw[level] (3.15,-.1) .. controls (3.6,-.25) and (4.4,-.25)  .. (4.85,-.1) node at (4,-.35) {$x_1$};
	\draw[level] (-4.9,.2) .. controls (-3,1) and (1,1) .. (2.9,.2) node at (-1,.95) {$x_2$};
	\draw[level] (-2.9,-.2) .. controls (-1,-1) and (3,-1) .. (4.9,-.2) node at (1,-.95) {$x_2$};
\end{tikzpicture}
\end{center}
\caption{Quiver representing the endomorphism algebra of $T'$ $(n=4)$}
\label{fig: Projective Quiver}
\end{figure}

Figures~\ref{fig: Hirzebruch Quiver} and \ref{fig: Projective Quiver} compare the endomorphism algebras of $T$ and $T'$ for $n=4$ (where there are implicit relations $x_ix_j = x_jx_i$ when both sides of the equality have the same head and tail).  Observe that removing the two central vertices from Figure~\ref{fig: Projective Quiver} gives exactly Figure~\ref{fig: Hirzebruch Quiver}.  In general, for $n \geq 2$ it is always the case that
\[
\op{End}(T) \cong \op{End}(\O_{\P(1:1:n)} \oplus \O_{\P(1:1:n)}(1)\oplus \O_{\P(1:1:n)}(n)\oplus \O_{\P(1:1:n)}(n+1)).
\]
This implies that for $n \geq 2$ there is a fully-faithful functor (see Theorem 2.29 of \cite{AKO1})
\begin{equation}
\op{MK}_n : \dbcoh{\mathbb{F}_n} \to \dbcoh{\P(1:1:n)}.\label{eqn: MK}
\end{equation}

In summary, while the Fukaya-Seidel categories of the LG mirrors are equivalent, the derived categories of coherent sheaves on the toric varieties are only related by a fully-faithful functor. It is natural to expect that the discrepancy comes from the fact that the anticanonical divisor on $\mathbb{F}_n$ is no longer nef for $n>2$.

Moreover, in the mirror, this discrepancy can be accounted for through a certain deformation.  Consider the family of isotopic LG models obtained by introducing the complex parameter $b$ in the equation
\[
w = x+y+\frac{1}{x} + \frac{b}{x^n y}.
\]
For $b\neq 0$ these functions yield equivalent Fukaya-Seidel categories, and have $n+2$ critical points, in bijection with the $n+2$ indecomposible summands of $T'$.  However, we have observed that $\dbcoh{\mathbb{F}_n}$ corresponds to $4$ of the indecomposible summands of $T'$.  This is explained by Lemma 5.4 of \cite{AKO1} which observes that if we let $b$ go to zero, $n-2$ of these critical points go off towards infinity while $4$ remain in a bounded region.  Figure~\ref{fig: critical points} is an image appearing in \cite{AKO1}.  It illustrates this phenomenon for $n=8$.   Therefore instead of the (naive) mirror, we can consider what we would call the homological mirror, which consists only of those critical fibers contained in the bounded region.  

One aim of this work is to understand this phenomena conceptually in terms of moduli spaces of Landau-Ginzburg models, as introduced in  \cite{DKK12}. This work also introduced the notion of ``radar screens" which gives a more general formalism 
for working with situations where critical loci of Landau-Ginzburg models assemble into regions upon a deformation. 

\begin{figure}[ht]

\begin{center}
\includegraphics[height=60mm]{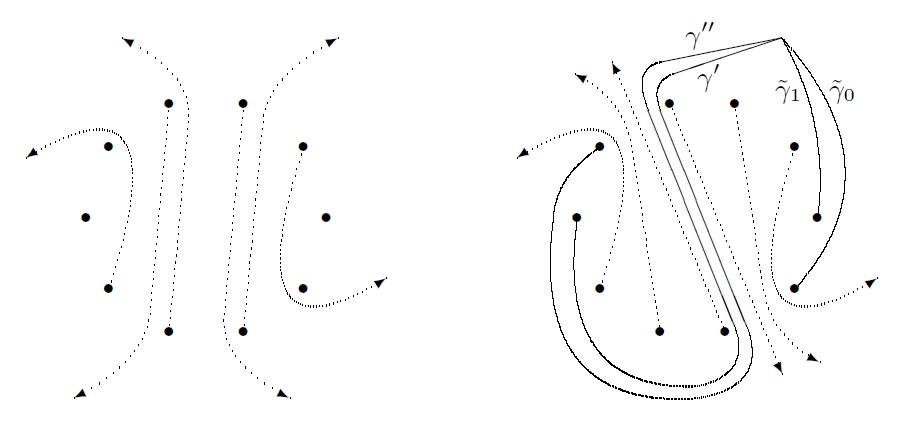}
\end{center}
\caption{The deformation $b \to 0$ $(n=8)$}
\label{fig: critical points}
\end{figure}

In general, mirror symmetry identifies complex moduli with K\"ahler moduli, and in the example above the complex deformation parameter $b$ is mirror to a deformation of the complexified K\"ahler class.  For a smooth toric variety $X$ the Hodge numbers $h^{2,0}$ and $h^{0,2}$ vanish, hence symplectic forms can be identified with $\R$-divisors $\op{H}^2(X, \R) \cong \op{Pic}(X)\otimes\mathbb{R} =: \op{Pic}_{\R}(X)$.  Furthermore, the GKZ fan, which parametrizes the various GIT quotients obtained from the Cox quotient construction, can also be identified with a fan supported on the pseudo-effective cone in $\op{Pic}_{\R}(X)$. In this way, degenerating the complex parameter of the mirror amounts to variation of GIT quotients (VGIT) on $X$. For the Hirzebruch surface, the two chambers of the GKZ fan correspond to $\P(1:1:n)$ and $\mathbb{F}_n$ as shown in Figure~\ref{fig: GKZ fan}.
\sidenote{{\color{red} Pseudo-effective = effective. Should we just say effective?}}

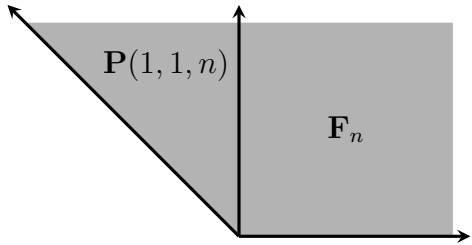
\begin{figure}[h]
\begin{center}
\begin{tikzpicture}
  [scale=.57, vertex/.style={circle,draw=black!100,fill=black!100,thick, inner sep=0.5pt,minimum size=0.5mm}, cone/.style={->,very thick,>=stealth}]
  \filldraw[fill=black!30!white,draw=white!100]
    (-5,5) -- (0,0) -- (5,0) -- (5,5) -- (-5,5);
  \draw[cone] (0,0) -- (5.4,0);
  \draw[cone] (0,0) -- (0,5.4);
  \draw[cone] (0,0) -- (-5.4,5.4);
  \node at (2.5,2.5) {$\mathbf{F}_n$};
  \node at (-1.7,4) {$\mathbf{P}(1,1,n)$};
\end{tikzpicture}
\end{center}
\caption{The GKZ fan of a Hirzebruch surface}
\label{fig: GKZ fan}
\end{figure}

In \cite{BFK12}, a precise relationship between the derived categories of quotient stacks coming from VGIT is given. It, in particular, reproduces the inclusion \ref{eqn: MK}.

In this paper, we aim to entirely describe this phenomenon for a general run of the toric Mori program for toric projective DM stacks, building on themes from \cite{DKK12} and \cite{BFK12}.  This includes a framework for Homological Mirror Symmetry for cases where the anticanonical divisor is not ample, nor even nef. 

Indeed, in \cite{DKK12}, the authors demonstrated that any (straight line) run of the toric Mori program degenerates the mirror so that critical points can be arranged into successive annuli.
This arrangement of critical points provides a semi-orthogonal decomposition of the Fukaya-Seidel category of this mirror.

\begin{introtheorem}[\cite{DKK12}] For every maximal degenerated LG mirror $\psi$ of a family $\psi (t)$ there is a
semi-orthogonal decomposition 
\[
\langle \overbrace{ \fs{\psi_1}, \ldots , \fs{\psi_1 }}^{-\mu_1},
\ldots , \overbrace{\fs{\psi_s } , \ldots, \fs{\psi_s }}^{-\mu_s} \rangle.
\]
of the Fukaya-Seidel category $\fs{\psi(t)}$ for $t \approx 0$. 
 Here the $\psi_i$ are regenerations of the degenerated LG model corresponding to each annulus and $-\mu_i$ are positive integers associated to the corresponding run of the toric Mori program.
\end{introtheorem}

On the other side of mirror symmetry, we recall the following result. 
\begin{introtheorem}[\cite{BFK12}] \label{thm: inclusion into nef}
 Let $X$ be a projective toric DM stack, equipped with a run of the toric Mori program. Starting from $X$, let $(W_1, \mu_1), \ldots, (W_s, \mu_s)$ be the wall data associated to the run. Then there exists a semi-orthogonal decomposition
\[
\langle \overbrace{\dbcoh{W_1}, \ldots, \dbcoh{W_1}}^{-\mu_1}, \ldots, \overbrace{\dbcoh{W_s}, \ldots, \dbcoh{W_s}}^{-\mu_s}  \rangle.
\]
of $\dbcoh{X}$. 
\end{introtheorem}

In conclusion, a (straight line) run of the toric Mori program for $X$ yields both a semi-orthogonal decomposition of $\dbcoh{X}$ and a semi-orthogonal decomposition of the Fukaya-Seidel category of the Landau-Ginzburg mirror of $X$.

From here on, let us say that a variety is {\bf nef-Fano} if it has $\mathbb{Q}$-factorial singularities and the anticanonical divisor $-K_X$ is big and nef. This condition is also often called weak Fano, usually with the additional assumption that the singularities are canonical. We will say that a Deligne-Mumford (or DM) stack is nef-Fano if its underlying coarse variety is nef-Fano. 

Given $X$ there is always a (straight line) run of the toric Mori program which contains a nef-Fano projective toric DM stack $X'$ as a phase (see Proposition~\ref{prop: existence}). This allows us to view the mirror to any $X$ as a Landau-Ginzburg model coming from a neighborhood of some collection of critical points in the mirror to $X'$. Therefore we propose that, like the homological mirror to $\mathbb{F}_n$, the homological mirror of any projective toric DM stack $X$ can be realized as a regeneration of the degenerated Landau-Ginzburg mirror  of $X'$ (in the sense of \cite{DKK12}).

With this picture in mind, we extend one half of Homological Mirror Symmetry for toric varieties as follows:

\begin{introconjecture} \label{conj: conjecture}
 Given, $\gamma$, a straight line run of the toric Mori program for a projective DM toric stack, $X$, and its associated maximal degeneration $\psi$ of an LG mirror pencil, let 
 \begin{align*}
  \op{Fuk}^{\rightharpoonup} (\psi(t) ) & = \langle \mathcal{T}_1, \ldots, \mathcal{T}_r \rangle \\
  \dbcoh{X} & = \langle \mathcal{S}_1, \ldots, \mathcal{S}_r \rangle
 \end{align*}
 be the semi-orthogonal decompositions associated to $\psi$ and $\gamma$.
 
 Then there exists an equivalence of triangulated categories
\begin{equation*} \Phi_\psi : \op{Fuk}^{\rightharpoonup} (\psi(t) ) \to \dbcoh{X} \end{equation*}
 which restricts to equivalences $\Psi_\psi : \mathcal{T}_i \to \mathcal{S}_i$ for all $1 \leq i \leq r$.
\end{introconjecture}

As evidence, we prove Conjecture \ref{conj: conjecture} in the following case: take $\P^2$, blow it up at a point, and then blow up at the two points on the exceptional locus corresponding to the two coordinate lines passing through the original point. Call the resulting space $X$.

\begin{introtheorem}
 Conjecture \ref{conj: conjecture} is true for a choice of run of the Mori program on $X$. In particular, Homological Mirror Symmetry holds for $X$.
\end{introtheorem}

\subsection{Relation with other work}

Homological Mirror Symmetry (HMS) was conceived for Calabi-Yau varieties and for toric varieties by Kontsevich \cite{KonHMS,KonCourse}.

Seidel proved HMS for the projective plane \cite{SeidelMore} and Auroux, Orlov, and  the fourth named author proved HMS for weighted projective planes \cite{AKO1} and del Pezzo surfaces (toric and non-toric) \cite{AKO2}. Independently, Ueda proved HMS for toric del Pezzo surfaces \cite{Ueda} and Ueda together with Yamazaki then proved it for toric orbifolds of del Pezzo surfaces \cite{UedYam}. Furthermore, Bondal and Ruan have announced a proof for weighted projective spaces of all dimensions \cite{BR}. 

Abouzaid proved HMS for smooth Fano toric varieties in \cite{Abouzaid} and, for general smooth projective toric varieties, provided an embedding
\begin{displaymath}
 \dbcoh{X} \to \op{Fuk}^{\rightharpoonup}(w)
\end{displaymath}
where $w$ is the Hori-Vafa Landau-Ginzburg mirror. We emphasize that when $X$ is not Fano, the image of Abouzaid's functor is not well understood in terms of mirror symmetry.  Fang, Liu, Treumann, and Zaslow proved a very similar statement via different means in \cite{FLTZ08}. \sidenote{\color{blue} this sentence is vague} A natural question to ask, given the results of \cite{Abouzaid} and \cite{FLTZ08}, is: for a non-nef Fano projective toric DM stack $X$, how does one naturally determine the subcategory of the Fukaya-Seidal category of the Hori-Vafa mirror which corresponds to the derived category of $X$? This paper proposes that following the toric Mori program through mirror symmetry gives the answer.

In the case of Picard rank one, Conjecture \ref{conj: conjecture} reduces to the usual statement of Homological Mirror Symmetry identically. For higher Picard rank, it includes the usual HMS statement, with additional structure.

While we are concerned with comparing the $B$-model on the toric stack to the $A$-model on the mirror LG model, there has also been work in the other direction, comparing the $A$-model on the toric stack to the $B$-model on the LG model. For $\P^2$, Chan and Leung showed how to produce matrix factorizations via Strominger-Yau-Zaslow \cite{SYZ} fibrations. In forthcoming work \cite{AFooooooooooooo}, Abouzaid, Fukaya, Oh, Ohta, and Ono advance further this direction of HMS. Abouzaid and Auroux have pointed out that our mirrors do not suffice for this other direction of mirror symmetry.

A similar result to Theorem \ref{thm: inclusion into nef} can be proven using Kawamata's pioneering work \cite{Kaw06}. The connection between the Mori program and derived categories also goes back to Kawamata in \cite{Kaw05, KawD-K} and extended in \cite{BFK12} and \cite{HL12}. 

On the other hand, as we will see, the moduli space of LG models (and it's compactification) is a crucial ingredient for the $A$-side of our approach.
Work in preparation of  Kontsevich, Pantev, and the fourth named author suggests a definition of a moduli space of LG models for essentially arbitrary potentials \cite{KKP13}.  
In another direction, it was pointed out to us by Keel that the results of \cite{Gross Keel Hacking} produce mirror families for a large class of rational surfaces which carry over many of the critical features we make use of for the case of toric mirrors.

In future work we intend to study the program suggested in this work for these rational surfaces.
In summary, we anticipate that the conjectural approach to proving HMS we outline here will find applicability for more general classes of varieties.

\subsection{Outline}
Let us briefly summarize the structure of this paper. We begin in Section~\ref{sec: toric minimal model} with a discussion of necessary background, including categorical and homological formalities and the general setup of the toric Mori program.  In Section~\ref{sec: B side}, we discuss the behavior of derived categories under variation of  GIT and the results of \cite{BFK12}, with an emphasis on the connections with Mori theory. Section~\ref{sec: A side} discusses the other side of mirror symmetry with an overview of \cite{DKK12} and general machinery for degenerations of mirror toric LG models. We conclude in Section~\ref{sec: program} with a conjectural program towards a proof of Conjecture~\ref{conj: conjecture}, and a complete execution of the program for a new example. 

\vspace{2.5mm}
\noindent \textbf{Acknowledgments:}
The authors have benefited immensely from conversations and correspondence with Denis Auroux, R. Paul Horja, Dmitri Orlov, Alexei Bondal, Kentaro Hori, Dragos Deliu, Umut Isik, Mohammed Abouzaid, Sean Keel, Dmitri Orlov, Alexander Kuznetsov, Yan Soibelman, and Maxim Kontsevich and would like to thank them all for their time, patience, and insight. 
The authors were funded by NSF DMS 0854977 FRG, NSF DMS 0600800, NSF DMS 0652633 FRG, NSF DMS 0854977, NSF DMS 0901330, FWF P 24572 N25, by FWF P20778 and by an ERC Grant. The first author was funded, in addition, by NSF DMS 0838210 RTG.
\vspace{2.5mm}

\section{Background}\label{section: Background}

\subsection{$A_{\infty}$-algebras and homological perturbation} \label{sec: perturbation}

We will assume that the reader is acquainted with the basics of $A_{\infty}$-algebras, including modules, morphisms, and derived categories. Excellent references for the material are \cite{Stasheff,Kel01,LH}. We begin by recalling the most important nontrivial equivalence relation on $A_{\infty}$-algebras.

\begin{definition}
 A morphism of $A_{\infty}$-algebras, $f: (A,m_n^A) \to (B,m^B_n)$, is called a \textbf{quasi-isomorphism} if the morphism on cohomology,
 \begin{displaymath}
  \op{H}^*(f): \op{H}^*(A,m^A_1) \to \op{H}^*(B,m^B_1),
 \end{displaymath}
 is an isomorphism.
\end{definition}

Quasi-isomorphic $A_{\infty}$-algebras have equivalent derived categories of $A_{\infty}$-modules. For a proof, see \cite{LH}.

\begin{lemma}
 If $f: (A,m_n^A) \to (B,m^B_n)$ is quasi-isomorphism, then there is an equivalence,
 \begin{displaymath}
  \op{D}(\op{perf} A) \cong \op{D}(\op{perf} B),
 \end{displaymath}
 between the derived categories of perfect modules.
\end{lemma}

Isomorphism classes of $A_{\infty}$-algebras can be quite large. Quasi-isomorphism classes will be even larger. We will be interested in locating a managable representative of a quasi-isomorphism equivalence class as we will need to perform full computations of the $A_{\infty}$-structure. 

One particularly useful class of $A_{\infty}$-algebras is the following. 

\begin{definition}
 An $A_{\infty}$-algebra, $(A,m_n^A)$, is called \textbf{minimal} if $m_1 = 0$. 
\end{definition}
\sidenote{{\color{red} Definition of formal?}}

Given a dg-algebra $(\widehat{A},d,m)$, let $A$ denote its cohomology algebra.  Suppose we have a collection of morphisms of vector spaces,
\begin{equation*}
 \includegraphics{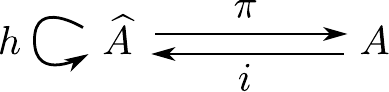}
\end{equation*}
satisfying
\[
 \pi \circ i = 1 \tand  i \circ \pi = 1-dh-hd \tand h^2 =0.
\]

Homological perturbation, see e.g. \cite{Ma04}, gives explicit higher products $m_n^A$ on $A$ and a quasi-isomorphism
between $\widehat{A}$ and $(A, m_n^A)$.  The formula for $m_n^A$  is given as sum over trees with $n$ leaves such that the valency of any internal vertex is $3$.  This is described now.

Given a tree, $T$, oriented from leaves to root with $n$ leaves, we write, $m^A_{T,n}$, for the composition of operators induced by the following rules:
\begin{itemize}
\item For each leaf compose with the map $i: A \ra  \widehat{A}$

\item Apply multiplication at every internal vertex

\item For each internal edge compose with the map $h:  \widehat{A} \ra  \widehat{A}$ 

\item Apply $\pi:  \widehat{A} \ra A$ at the root

\end{itemize}

The $m_n^A$ are given by the formula
\begin{equation}
m_n^A = \sum_T (-1)^{\vartheta(T)} m^A_{T,n}
\end{equation}
where the sum is taken over all planar trees with $n$-leaves. For the precise signs, $\vartheta(T)$, see \cite{Ma04}. We will only need the formula for $n=3$ which we display graphically.

\begin{center}
 \begin{tikzpicture}

  \node at (-6.2,0) {$m^A_3 \ =$};

  \node at (-3.5,0) {\begin{tikzpicture}[scale=0.7]
  \node[draw,circle] (a) at (0,2) {$\pi$};
  \node[draw,circle] (b) at (0,0) {$m$};
  \node[draw,circle] (c) at (-1,-1) {$h$};
  \node[draw,circle] (d) at (-2,-2) {$m$};
  \node[draw,circle] (e) at (4,-4) {$i$};
  \node[draw,circle] (f) at (0,-4) {$i$};
  \node[draw,circle] (g) at (-4,-4) {$i$};
  
  \draw (a) -- (b);
  \draw (b) -- (c);
  \draw (b) -- (e);
  \draw (c) -- (d);
  \draw (d) -- (g);
  \draw (d) -- (f);
 \end{tikzpicture}
 };
 
  \node at (0,0) {$-$};
 
  \node at (3.5,0) {\begin{tikzpicture}[scale=0.7]
  \node[draw,circle] (a) at (0,2) {$\pi$};
  \node[draw,circle] (b) at (0,0) {$m$};
  \node[draw,circle] (c) at (1,-1) {$h$};
  \node[draw,circle] (d) at (2,-2) {$m$};
  \node[draw,circle] (e) at (-4,-4) {$i$};
  \node[draw,circle] (f) at (0,-4) {$i$};
  \node[draw,circle] (g) at (4,-4) {$i$};
  
  \draw (a) -- (b);
  \draw (b) -- (c);
  \draw (b) -- (e);
  \draw (c) -- (d);
  \draw (d) -- (g);
  \draw (d) -- (f);
 \end{tikzpicture}
 };
 \end{tikzpicture}
\end{center}

In terms of equations, we have 
\begin{equation} \label{eq: m3}
 m^A_3(a,b,c) = \pi(m(h(m(i(a),i(b))),i(c))) - (-1)^{|a|} \pi(m(i(a),h(m(i(b),i(c)))))
\end{equation}
where $|a|$ is the degree of $a$.

\subsection{Exceptional collections, quivers, and equivalences}
Given a triangulated category, we often wish to describe it by decomposing it into more manageable subcategories. To this end, a common structure appearing in \cite{BFK12} and \cite{DKK12} is that of a semi-orthogonal decomposition. In this section we quickly review this notion and focus on the special case of a full exceptional collection. Standard references for the material are \cite{Bon,BK,Kel01}.

\begin{definition}\label{def:SO}
 A \textbf{semi-orthogonal decomposition} of a triangulated category, $\mathcal T$, is a sequence of full triangulated subcategories, $\mathcal A_1, \dots ,\mathcal A_m$, in $\mathcal T$ such that $\mathcal A_i \subset \mathcal A_j^{\perp}$ for $i<j$ and, for every object $T \in \mathcal T$, there exists a diagram:
 \begin{center}
 \begin{tikzpicture}[description/.style={fill=white,inner sep=2pt}]
 \matrix (m) [matrix of math nodes, row sep=1em, column sep=1.5em, text height=1.5ex, text depth=0.25ex]
 {  0 & & T_{m-1} & \cdots & T_2 & & T_1 & & T   \\
   & & & & & & & &  \\
   & A_m & & & & A_2 & & A_1 & \\ };
 \path[->,thick]
  (m-1-1) edge (m-1-3) 
  (m-1-3) edge (m-1-4)
  (m-1-4) edge (m-1-5)
  (m-1-5) edge (m-1-7)
  (m-1-7) edge (m-1-9)

  (m-1-9) edge (m-3-8)
  (m-1-7) edge (m-3-6)
  (m-1-3) edge (m-3-2)

  (m-3-8) edge node[sloped] {$ | $} (m-1-7)
  (m-3-6) edge node[sloped] {$ | $} (m-1-5) 
  (m-3-2) edge node[sloped] {$ | $} (m-1-1)
 ;
 \end{tikzpicture}
 \end{center}
 where all triangles are distinguished and $A_k \in \mathcal A_k$. We shall denote a semi-orthogonal decomposition by $\langle \mathcal A_1, \ldots, \mathcal A_m \rangle$. 
 If the subcategories in a semi-orthogonal decomposition are the essential images of fully-faithful functors, $\Upsilon_i: \mathcal A_i \to \mathcal T$, we shall also denote a semi-orthogonal decomposition by
 \begin{displaymath}
  \langle \Upsilon_1, \ldots, \Upsilon_m \rangle.
 \end{displaymath}

 Assume that $\mathcal T$ is linear over a field, $k$. Let $E_1,\ldots,E_n$ be objects of $\mathcal T$. We say that $E_1,\ldots,E_n$ is an \textbf{exceptional collection} if 
 \begin{displaymath}
  \op{Hom}_{\mathcal T}(E_i,E_i[l]) = \begin{cases} k & \text{if } l = 0 \\ 0 & \text{otherwise} \end{cases}
 \end{displaymath}
 for all $i$ and
 \begin{displaymath}
  \op{Hom}_{\mathcal T}(E_j,E_i[l]) = 0, \forall j > i, \forall l.
 \end{displaymath}
 We say that $E_1,\ldots,E_n$ is \textbf{full} if $E_1,\ldots,E_n$ generates $\mathcal T$. We say $E_1,\ldots,E_n$ is a \textbf{strong} exceptional collection if 
 \begin{displaymath}
  \op{Hom}_{\mathcal T}(E_i,E_j[l]) = 0, l \not = 0, \forall i,j.
 \end{displaymath}
\end{definition}

Possession of a full exceptional collection allows for a simplified presentation of an (enhanced) triangulated category. Let $E_1,\ldots,E_n$ be a full exceptional collection in $\mathcal T$. Then to $E_1,\ldots,E_n$ we can associate a graded quiver with relations, $(Q,R)$. The vertex set of $Q$ is the set $\{1,2,\ldots,n\}$. Take a minimal set, $S$, of generators for the graded endomorphism algebra
\begin{displaymath}
 A = \bigoplus_{l \in \Z} \op{Hom}_{\mathcal T} (\oplus_{j=1}^n E_j, \oplus_{j=1}^n E_j[l])
\end{displaymath}
consisting of elements of the morphism spaces $\op{Hom}_{\mathcal T}(E_i, E_j[l])$. The arrows from $i$ to $j$ in $Q$ will be the elements of $S$ lying in $\op{Hom}_{\mathcal T}(E_i, E_j[l])$ for some $l \in \Z$. The degree of the arrow will be $l$. 

To determine $R$, we note that, as we used a set of generators for $A$, there is a natural surjection
\begin{displaymath}
 \pi: \C Q \to A
\end{displaymath}
from the path algebra of $Q$ to $A$. The ideal of relations, $R$, is the kernel of $\pi$.

One would hope that $\mathcal T$ is equivalent to the derived category of $(Q,R)$ or equivalently to the derived category of graded $A$-modules. However, this hope is too naive. One must assume that $\mathcal T$ possesses some additional structure, an enhancement \cite{BK}, to get something close to this. 

\begin{proposition}
 If $\mathcal T$ possesses an enhancement, then $A$ can be given the structure of a minimal $A_{\infty}$-algebra, also denoted by $A$, such that there is an equivalence
\begin{displaymath}
 \mathcal T \cong \op{D}(\op{perf} A)
\end{displaymath}
 between $\mathcal T$ and the derived category of perfect $A_{\infty}$-modules over the $A_{\infty}$-algebra, $A$.
\end{proposition}
\noindent For a proof of the proposition see \cite{Bon,Kel94}.

The above proposition allows us to determine $\mathcal T$ up to equivalence from a finite set of data: the graded quiver $Q$, the relations, $R$, and the higher $A_{\infty}$ maps,
\begin{displaymath}
 m^A_n: A^{\otimes n} \to A, \ n \geq 3.
\end{displaymath}

Let us recall a well-known example: the Beilinson exceptional collection \cite{Bei}
\begin{displaymath}
 \mathcal O_{\mathbb{P}^n}, \mathcal O_{\mathbb{P}^n}(1),\ldots, \mathcal O_{\mathbb{P}^n}(n)
\end{displaymath}
is a strong full exceptional collection for the derived category $\dbcoh{\mathbb{P}^n}$. The graded quiver we get is pictured in Figure \ref{fig: Beilinson quiver}. All arrows are in degree $0$. 

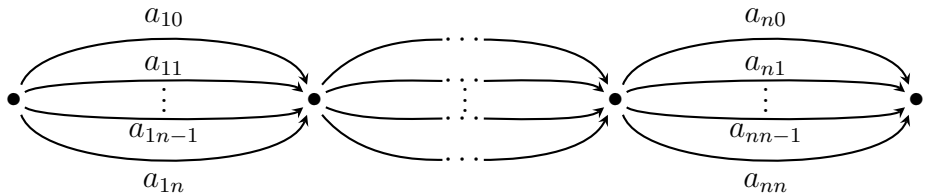
\begin{figure}[ht]
\begin{center}
\begin{tikzpicture}[scale=1,level/.style={->,>=stealth,thick}]
	\node (a) at (-6,0) {$\bullet$};
	\node (b) at (-4,.1) {$\vdots$};
	\node (c) at (-2,0) {$\bullet$};
	\node at (0,.1) {$\vdots$};
	\node at (0,.25) {$\cdots$};
	\node at (0,.8) {$\cdots$};
	\node at (0,-.25) {$\cdots$};
	\node at (0,-.8) {$\cdots$};
	\node (d) at (2,0) {$\bullet$};
	\node (e) at (4,.1) {$\vdots$};
	\node (f) at (6,0) {$\bullet$};
	\draw[level] (-5.85,.1) .. controls (-5.6,.3) and (-2.4,.3)  ..  (-2.15,.1) node at (-4,.45) {$a_{11}$};
	\draw[level] (-5.85,-.1) .. controls (-5.6,-.3) and (-2.4,-.3)  ..  (-2.15,-.1) node at (-4,-.45) {$a_{1n-1}$};
	\draw[level] (-5.9,.2) .. controls (-5.6,1) and (-2.4,1) .. (-2.1,.2) node at (-4,1.1) {$a_{10}$};
	\draw[level] (-5.9,-.2) .. controls (-5.6,-1) and (-2.4,-1) .. (-2.1,-.2) node at (-4,-1.1) {$a_{1n}$};
	
	\draw[level] (2.15,.1) .. controls (2.4,.3) and (5.6,.3)  ..  (5.85,.1) node at (4,.45) {$a_{n1}$};
	\draw[level] (2.15,-.1) .. controls (2.4,-.3) and (5.6,-.3)  ..  (5.85,-.1) node at (4,-.45) {$a_{nn-1}$} ;
	\draw[level] (2.1,.2) .. controls (2.4,1) and (5.6,1) .. (5.9,.2) node at (4,1.1) {$a_{n0}$} ;
	\draw[level] (2.1,-.2) .. controls (2.4,-1) and (5.6,-1) .. (5.9,-.2) node at (4,-1.1) {$a_{nn}$};
	
	\draw[thick] (-1.85,.1) .. controls (-1.4,.3) and (-0.6,.25)  ..  (-.3,.25);
	\draw[level] (.25,.25) .. controls (0.4,.25) and (1.6,.3)  ..  (1.85,.1);
	
	\draw[thick] (-1.85,-.1) .. controls (-1.4,-.3) and (-0.6,-.25)  ..  (-.3,-.25);
	\draw[level] (.25,-.25) .. controls (0.4,-.25) and (1.6,-.3)  ..  (1.85,-.1);
	
	\draw[thick] (-1.9,.2) .. controls (-1.4,.8) and (-0.6,.8)  ..  (-.3,.8);
	\draw[level] (.25,.8) .. controls (0.4,.8) and (1.6,.8)  ..  (1.9,.2);
	
	\draw[thick] (-1.9,-.2) .. controls (-1.4,-.8) and (-0.6,-.8)  ..  (-.3,-.8);
	\draw[level] (.25,-.8) .. controls (0.4,-.8) and (1.6,-.8)  ..  (1.9,-.2);

\end{tikzpicture}
\end{center}
\caption{Beilinson quiver}
\label{fig: Beilinson quiver}
\end{figure}

The ideal of relations is generated by
\begin{displaymath}
 a_{i+1j}a_{ik} - a_{i+1k}a_{ij}
\end{displaymath}
for $0 \leq i \leq n-1$ and $0 \leq j,k \leq n$. It expresses commutativity of the polynomial ring. Finally, the higher $A_{\infty}$-structure maps are trivial for degree reasons. 

Thus, whenever we can find an enhanced triangulated category, $\mathcal T$, and $(n+1)$-objects, $T_0,\ldots,T_n$ of $\mathcal T$ such that the endomorphism algebra of $\oplus_{i=0}^n T_i$ is the path algebra of the Beilinson quiver modulo the commutativity relations, we automatically get a fully-faithful functor,
\begin{displaymath}
 \dbcoh{\mathbb{P}^n} \to \mathcal T.
\end{displaymath}
If we can show that $\mathcal T$ generates, we will get an equivalence. Note that this phenomenon explains why the quivers from Figures~\ref{fig: Hirzebruch Quiver} and \ref{fig: Projective Quiver} yield the fully-faithful functor $\op{MK}_n$ in \eqref{eqn: MK}.

%

\subsection{The toric Mori program} \label{sec: toric minimal model}

In this section, we briefly review the toric Mori program, also called the toric minimal model program, which originated in the work of Reid \cite{Reid83}; standard references are the books of Cox, Little, Schenck \cite[Chapters 14-15]{CLS} and Matsuki \cite[Chapter 14]{Matsuki02}. The subject has interpretations in terms of birational geometry, variation of Geometric Invariant Theory (VGIT), and combinatorial modifications of toric fans. It can be developed in each language in isolation. 

We will focus mainly on VGIT. We choose this direction for its immediate relation with the work of \cite{BFK12} and for its simplicity of presentation. For the birational geometer interested in comparing our presentation through VGIT to their more familiar perspective of flips, flops, and Mori contractions, we recommend reviewing \cite{HK2}. For those more comfortable with the combinatorial viewpoint, we recommend \cite{CLS}. However, to make connections with \cite{DKK12}, we will also need to recall some combinatorial properties. We begin by recalling the general setup of VGIT \cite{DH98,Tha96}. 

Let $G$ be a linearly reductive group acting on a variety $X$ which is equipped with a $G$-equivariant line bundle $\mathcal L$. Following Mumford \cite{MFK}, we consider three subsets of $X$:
\begin{definition} \label{definition: semi-stable,stable,unstable loci}
\begin{align*}
 X^{\op{ss}}(\mathcal L) & = \{x \in X \ | \ \exists f \in \op{H}^0(X,\mathcal L^n)^G \text{ with } n > 0, f(x) \not = 0, \tand X_f \text{ affine} \} \\
 X^{\op{s}}(\mathcal L) & = \{x \in X^{\op{ss}}(\mathcal L) \ | \ G \cdot x \text{ is closed in } X^{\op{ss}}(\mathcal L) \text{ and }G_x \text{ is finite}\} \\
 X^{\op{us}}(\mathcal L) & = X \setminus X^{\op{ss}}(\mathcal L).
\end{align*}
 where $X_f = \{x\in X \ | \ f(x) \neq 0\}$. The subsets $X^{\op{ss}}(\mathcal L), X^{\op{s}}(\mathcal L)$, and $X^{\op{us}}(\mathcal L)$, are called the \textbf{semi-stable, stable}, and \textbf{unstable locus}, respectively, of $X$. Each subset is naturally a subvariety of $X$.
\end{definition}

\begin{definition} \label{definition: GIT quotient}
With notation as above, the \textbf{GIT quotient} of $X$ by $G$ with respect to $\mathcal L$ is the quotient stack, $[X^{\op{ss}}(\mathcal L)/G]$. We denote the GIT quotient by $X \modmod{\mathcal L} G$. 
\end{definition}

The reader should note that this is {\em not} the definition of a GIT quotient given in \cite{MFK}, as we work with stacks.  Indeed, the GIT quotient appearing in \cite{MFK} is the coarse moduli space of the GIT quotient appearing here, assuming that the stack is Deligne-Mumford. 

The GIT quotient depends on the choice of $\mathcal L$. Exploring the relationship between the quotient stacks resulting from different choices of $\mathcal L$ has deep ties with birational geometry. Let $\op{Pic}^G(X)$ be the group of $G$-equivariant line bundles on $X$. The associated vector space, $\op{Pic}^G_{\R}(X)$, naturally forms a parameter space for GIT quotients. We say that $\mathcal L$ is \textbf{$G$-effective} if $\mathcal L$ is ample and $X^{\op{ss}}(\mathcal L) \not = \emptyset$ and we let $\op{Eff}^G_{\R}(X)$ denote the cone of $G$-effective equivariant line bundles.

\begin{definition}
 The \textbf{GIT fan} for the action of $G$ on $X$ is a fan in $\op{Pic}^G_{\R}(X)$ whose support is $\op{Eff}^G_{\R}(X)$. The GIT fan is the unique fan characterized by the following property: the interiors of the cones of the GIT fan are exactly the subsets of constant semi-stable locus.
\end{definition}

In general, the GIT fan may not exist. However, in many cases, it does. For example, if $X$ is proper \cite{DH98,Tha96,Res} it exists. This implies that it exists for linear $G$-actions on $X = \mathbb{A}^n$ since we can compactify $\mathbb{A}^n$ to $\mathbb{P}^n$ and extend the action correspondingly. 

We now leave the general setting of GIT to focus on toric setting: $X=\mathbb{A}^n$ and $G$ is a subset of $\mathbb{G}_m^n$. Affine space, $\mathbb{A}^n$, has trivial Picard group so the set of $G$-equivariant line bundles is in bijection with characters of the group $G$, denoted by $\widehat{G}$. Given a character, $\chi$, of $G$, we denote the corresponding $G$-equivariant line bundle by $\mathcal O(\chi)$ and the corresponding GIT quotient by $\mathbb{A}^n \modmod{\chi} G$.

\begin{definition}
 When $G \subseteq\mathbb{G}_m^n$ is a subgroup of an algebraic torus, we shall call a GIT quotient, $\mathbb{A}^n \modmod{\chi} G$, of affine space $\mathbb{A}^n$ by $G$, a \textbf{toric stack}. We will say that the quotient is \textbf{projective} if $\C[x_1,\ldots,x_n]^G = \C$. If the quotient is represented by a scheme, we shall call the quotient a \textbf{toric variety}.
\end{definition}

\begin{remark}
 We have already mentioned that the GIT fan for toric actions on affine space exists. It is straightforward to see that a toric stack is Deligne-Mumford if and only if $\chi$ lies in the interior of a top dimensional cone of the GIT fan. This is equivalent to the underlying coarse toric variety being simplicial.
\end{remark}

\begin{remark}
 Note that, in general, a GIT quotient of affine space is automatically projective over an affine variety. In particular, if the quotient is proper, then it is projective; this occurs when $\C[x_1,\ldots,x_n]^G = \C$.  However, the most common definition of a toric stack proceeds through a stacky fan, as in \cite{BCS} or the more general construction in \cite{GS}. These approaches also include non-projective toric stacks. At least in the case of varieties, our more restrictive definition has appeared elsewhere in the literature, see \cite{MS} as an example. 
\end{remark}


GIT fans for toric varieties have deeper roots. For toric varieties, when $\mathbb{A}^n$ is the spectrum of the Cox ring, the GIT fan was first considered by Oda-Park \cite{odapark}, motivated by a dual construction for polytopes introduced by Gel'fand, Kapranov, and Zelevinsky \cite{GKZ}. As such, it often goes by other names, including the \textbf{GKZ fan} or \textbf{secondary fan}, and it admits a purely combinatorial description, again see \cite{CLS,Matsuki02} for details.

We briefly recall the structure of the GKZ fan for toric quotients, again see \cite{DH98,Tha96} for a treatment of the general case. The inclusion, $G \subset \mathbb{G}_m^n$ yields an exact sequence of abelian groups 
\begin{equation} \label{eqn: characters exact sequence}
 0 \to M \overset{\delta}{\to} \mathbb{Z}^n \overset{\gamma}{\to} \widehat{G} \to 0,
\end{equation}
of abelian groups, where $M$ denotes the kernel of $\gamma$. Tensoring with $\mathbb{R}$, we obtain a sequence of vector spaces \begin{equation*}
 0 \to M_{\mathbb{R}} \to \mathbb{R}^n \overset{\gamma_{\mathbb{R}}}{\to} \widehat{G}_{\mathbb{R}} \to 0.
\end{equation*} Let $\beta_i := \gamma(e_i)$ with $e_1,\ldots,e_n$ the standard basis of $\mathbb{Z}^n$.  Let $C_{\beta}$ be the cone in $\widehat{G}_{\mathbb{R}}$ spanned by $\beta_1,\ldots,\beta_n$. It is easy to check that the characters in $C_{\beta}$ are exactly those whose corresponding GIT quotients are non-empty.

\begin{remark} In addition to the combinatorial descriptions, this fan also admits descriptions in terms of birational geometry. Recall that we can identify $\widehat{G}_{\mathbb{R}}$ with $\op{Pic}_{\R}(X)$ when $X$ is a projective simplicial toric variety representing the corresponding quotient. In turn, $\op{Pic}_{\R}(X)$ coincides in this case with $N^{1}(X)$, the space of divisors modulo numerical equivalence. The cone $C_\beta$ then coincides with the cone of effective divisors $\op{Eff}(X)$. This cone admits a fan called the Mori fan, where  divisors in the relative interior of the same chamber define the same rational contraction; see e.g. \cite{HK2} for a thorough discussion which in particular proves the equivalence of this fan with the GIT fan. We will occasionally use the standard facts that the cone $\op{Amp}(X)\subset N^{1}(X)$ of ample divisors is the relative interior of the closed cone of nef divisors $\op{Nef}(X)$, that the relative interior of 
the cone $\op{Eff}(X)$ is the cone of big divisors $\op{Big}(X)$, and that in the toric case $\op{Eff}(X)$ is closed. 
\end{remark}


The maximal cones in the GIT fan are called {\bf{chambers}}, and the codimension $1$ cones are called {\bf{walls}}.  
We say that two smooth toric projective DM stacks are {\bf{neighbors}}, if there is a GKZ fan $\mathcal F_{\op{GKZ}}$ such that the two stacks are GIT quotients corresponding to adjacent chambers separated by a wall. Motivated by string theory terminology, given a fixed GKZ fan, if $X$ occurs as a GIT quotient corresponding to some maximal chamber, we say $X$ is a {\bf phase} of $\mathcal F_{\op{GKZ}}$. 

\sidenote{{\color{red} Using $\Sigma$ for the fan and the maximal cones is a bit confusing.} Changed to $\mathcal F$ {\color{red}Should we have a $\tau$ subscript on $W$ and $\mu$?} Naw.}

Now, let $C_+$ and $C_-$ be two adjacent chambers in a GKZ fan. Let $\tau$ be the wall defined by the intersection of $C_+$ and $C_-$. The wall $\tau$ spans a linear hyperplane, $H$. Let $\lambda: \mathbb{G}_m \to G$ be a primitive one-parameter subgroup whose kernel, once tensored with $\mathbb{R}$, defines $H$. Furthermore, choose $\lambda$ so that the pairing of $\lambda$ with any character in $C_+$ is nonnegative. Set 
\begin{displaymath}
 G_{\lambda} := G/\lambda(\mathbb{G}_m).
\end{displaymath}
Let $\alpha$ be a character in the interior of $\tau$. Finally, consider the GIT quotient 
\begin{displaymath}
 W := (\mathbb{A}^n)^{\lambda} \modmod{\alpha} G_{\lambda}
\end{displaymath}
of the fixed locus of $\lambda$ by $G_{\lambda}$ via $\mathcal O(\alpha)$. 

Set 
\begin{displaymath}
 \mu := -\langle \lambda, K_{[\mathbb{A}^n/G]} \rangle
\end{displaymath}
where $K_{[\mathbb{A}^n/G]} = -\sum_i \beta_i$.

\begin{definition}
 For a wall $\tau$ separating adjacent chambers $C_+$ and $C_-$ in a GIT fan, we call $(W, \mu)$ the {\bf{wall data}} associated to $\tau$.
\end{definition}


We are interested not in only in individual walls, but sequences of walls and the corresponding sequences of maximal chambers. Of particular importance are sequences which comes from a sequence of Mori operations (i.e. flops, directed flips, divisorial contractions, and Mori fibrations). To this end, fix a projective toric DM stack $X$ and consider a continuous path $\gamma: [0,1] \to \widehat{G}_\mathbb{R}$ such that:
 \begin{itemize}\item$\gamma(0)$ lies in the ample cone of $X$,
  \item$\gamma(1)$ lies outside the support of the GKZ fan,
   \item$\gamma$ does not pass through any cones in the GKZ fan of codimension 2.
   \end{itemize} 
 For any $0 < t < 1$, if $\gamma(t)$ lies in a wall, then we can consider the corresponding adjacent chambers $C_+$ and $C_-$ where $\gamma(t-\epsilon) \in C_-$ and $\gamma(t + \epsilon) \in C_+$ for sufficiently small $\epsilon$.  This allows us to define a sequence of wall data along the path.
\begin{definition}
A path, $\gamma$, as above, is called a {\bf{run of the toric Mori program}} if whenever $\gamma$ intersects a wall $\tau$ one has  $\mu \leq 0$ for the wall data $(W, \mu)$. If $\gamma$ is run of the toric Mori program and is a line segment with respect to the Euclidean topology on $\widehat{G}_\mathbb{R}$ we call $\gamma$ a {\bf straight line run}.
\end{definition}

We will frequently make use of the following elementary observation, which essentially states that any phase can be reached by a straight line run. 

\begin{lemma}\label{lemma: straight} Let $X$ be a projective toric DM stack with $GKZ$ fan $\mathcal F_{\op{GKZ}}$.  Let $Y$ be any phase of $\mathcal F_{\op{GKZ}}$ with corresponding chamber $C_Y$. Fix any character $\chi$ in the ample cone of $X$.  There is a line segment $\gamma$  with respect to the Euclidean topology on $\widehat{G}_\mathbb{R}$ satisfying the three conditions above such that $\gamma (0)=\chi$ and $\gamma (a)$ lies in $C_Y$ for some $0 < a < 1$. \end{lemma}
\begin{proof}
Fix the Euclidean metric on $\widehat{G}_{\mathbb{R}}$ and take $\epsilon > 0$ sufficiently small so that the sphere $S_{\epsilon}(\chi)$ of radius $\epsilon$ centered at $\chi$ is contained in the ample cone. Any point $\alpha\in S_{\epsilon}(\chi)$ determines a unique ray emanating from $\chi$.  Consider the collection of such points which determine a ray which eventually passes outside the pseudo-effective cone.  As $X$ is projective, the psedo-effective cone is strongly convex.  Hence, the collection of all such points with a ray hitting the interior of $C_Y$ is a non-empty open subset of the the sphere as $C_Y$ is of maximum dimension. However, the locus of all points whose ray intersects a codimension $2$ cone of $\mathcal F_{\op{GKZ}}$ is a cell complex of real codimension at least 1. 
\end{proof}

As discussed in the introduction, one of the aims of this paper is to study Homological Mirror Symmetry for smooth toric DM stacks which are not Fano (or nef-Fano). This following simple observation lets us always reach this case by a straight line run.  

\begin{proposition}\label{prop: existence}
 Let $X$ be a projective toric DM stack. Then there exists a nef-Fano projective toric DM stack $X'$ and a straight line run $\gamma $ in the GKZ fan of $X$ which starts at the phase $X'$ and passes through the phase $X$ and its way to the exterior of the pseudo-effective cone. Moreover, if the anticanonical divisor $-K_X$ is contained in the relative interior of a maximal chamber of the GKZ fan of $X$, then $X'$ may be taken to be Fano. 
\end{proposition}
\begin{proof}
Let $C_{-K}$ be any chamber which contains $-K_X$ (possibly in its boundary), and let $X'$ be GIT quotient corresponding to this chamber. By the lemma above, let $\gamma$ be a line segment starting in the relative interior of $C_{-K}$, passing through the ample cone of $X$, and eventually exiting the pseudo-effective cone.

 To see that this path is a run of the toric Mori program, consider a value $0 < a <1$ such that $\gamma(a)$ lies in a wall of the GKZ fan of $X$.  By definition, 
\[  \mu = -\langle \lambda, K_{[\mathbb{A}^n/G]} \rangle \]
where $\lambda$ is a one-parameter subgroup with $\langle \lambda,  \gamma(a + \epsilon) \rangle > 0$ and   $\langle \lambda,  \gamma(a -\epsilon) \rangle < 0$ for some sufficiently small value of $\epsilon >0$.  Now, since our path is a straight line we have $\langle \lambda, \gamma(a -\epsilon) \rangle < 0$ for any $\epsilon > 0$.  In particular, $\langle \lambda, \gamma(0) \rangle < 0$ and remains negative on the relative interior of $C_{-K}$.  By continuity $\langle \lambda,  v \rangle \leq 0$ for any $v \in C_{-K}$.  Hence $\mu = -\langle \lambda, K_{[\mathbb{A}^n/G]} \rangle \leq 0$ as desired.
 
Recall that the anticanonical divisor on a projective toric variety is automatically effective and big.  To see that it is also nef, notice that the GKZ fan of $X'$ is a projection of the GKZ fan of $X$ (possibly the identity projection), and the nef cone of $X'$ is the image of $C_{-K}$ under this projection
, from which the proposition immediately follows.\end{proof}

When we consider toric mirror symmetry, it will be natural to consider an enlargement of the GKZ fan which is complete. The following proposition is well-known. 

\begin{proposition}\label{prop: enlarged} Let $X$ be a projective toric DM stack. Then the GKZ fan of $X$ is naturally a subfan of the GKZ fan of the total space of the anticanonical bundle $\op{Tot}(-K_X)$, and moreover the GKZ fan of $\op{Tot}(-K_X)$ is complete. 
\end{proposition} 


Following \cite{AGM}, we call the GKZ fan of the anticanonical bundle of $X$ the {\bf enlarged GKZ fan} of $X$. That this enlarged fan is relevant for mirror symmetry is reasonable, as $\op{Tot}(-K_X)$ is log Calabi-Yau.

\section{The $B$-side} \label{sec: B side}

\subsection{The toric Mori program and derived categories}

The purpose of the Mori program is to take a projective variety, $X$, and, by a controlled sequence of birational morphisms, produce a ``simplest'' birational model of $X$. As birational geometry and derived categories are intimately tied, it is natural to expect that, if the Mori program says to replace $X$ by $X'$, then the derived categories of $X$ and $X'$ should be closely related. The most natural language to express such a relation is that of semi-orthogonal decompositions.

A conjecture of Kawamata \cite{KawD-K} gives a precise condition for $\dbcoh{X}$ and $\dbcoh{X'}$ to be related by a semi-orthogonal decomposition whenever $X$ and $X'$ are birational.

\begin{definition}
 Let $X$ and $X'$ be smooth projective varieties. If there are birational morphisms,
 \begin{center}
  \begin{tikzpicture}
   \node (a) at (0,1) {$Z$};
   \node (b) at (-1,0) {$X$};
   \node (c) at (1,0) {$X'$};
   \node at (-.65,.65) {$f$};
   \node at (.65,.65) {$g$};
   \draw[->,thick] (a) -- (b) ;
   \draw[->,thick] (a) -- (c) ;
  \end{tikzpicture}
 \end{center}
 such that $f^*\omega_X \otimes g^*\omega_{X'}^{-1}$ is effective, we say that $X$ \textbf{$K$-dominates} $X'$.  
\end{definition}

\begin{conjecture}[Kawamata]  \label{conj: DK}
 If $X$ $K$-dominates $X'$, then there is a fully-faithful functor $\dbcoh{X'}\rightarrow\dbcoh{X}$, hence (automatically by \cite{BV}) a semi-orthogonal decomposition of $\dbcoh{X}$ with $\dbcoh{X'}$ as a component. 
\end{conjecture}

As discussed in Section~\ref{sec: toric minimal model}, the toric Mori program is completely understood. Using this understanding, Kawamata proved Conjecture \ref{conj: DK} in \cite{Kaw05,Kaw06,Kaw12} for smooth projective toric varieties. Moreover, if $X'$ is obtained from $X$ via a step in the toric Mori program, Kawamata exhibited a semi-orthogonal decomposition of $\dbcoh{X}$ with $\dbcoh{X'}$ as a component and showed that the complementary component in the semi-orthogonal decomposition possessed a full exceptional collection. In particular, 

\begin{theorem}[Kawamata] \label{thm: Kaw exc coll}
 Any projective DM toric stack admits a full exceptional collection (not necessarily strong or of line bundles). 
\end{theorem}

Again, as discussed in Section~\ref{sec: toric minimal model}, the toric Mori program for a smooth projective toric variety $X$ can be embedded in VGIT by using a presentation of $X$ as a GIT quotient, $\mathbb{A}^n \modmod{\chi} G$. Hence, the results from \cite{BFK12} on the relationship of derived categories of varieties, or stacks, coming from VGIT can be brought to bear on the problem.

The following result is Theorem 5.2.1 in \cite{BFK12}:

\begin{theorem} \label{theorem: SOD of derived categories for toric variations}
Let $X$ and $Y$ be projective DM toric stacks which are neighbors and let $(W, \mu)$ be the associated wall data for a path passing from the chamber for $X$ to the chamber for $Y$.  For any choice of $d \in \mathbb{Z}$ the following statements are true:
 
\begin{enumerate}
 \item If $\mu > 0$, there exists fully-faithful functors,
\begin{displaymath}
 \Phi_d : \dbcoh{ X} \to \dbcoh{Y}
\end{displaymath}
 and
\begin{displaymath}
 \Upsilon_-: \dbcoh{W} \to \dbcoh{Y},
\end{displaymath}
 and a semi-orthogonal decomposition, with respect to $\Phi_d, \Upsilon_+$,
\begin{displaymath}
 \dbcoh{ Y} = \langle \dbcoh{W}(-\mu - d + 1), \ldots, \dbcoh{W}(-d), \dbcoh{X } \rangle,
\end{displaymath}
 with $\dbcoh{W}(l)$ the image of $\Upsilon_-$ twisted by a character of $\lambda$-weight $l$. 
\item If $\mu = 0$, there is an equivalence
\begin{displaymath}
 \Phi_d : \dbcoh{ X } \to \dbcoh{ Y}.
\end{displaymath}
\item If $\mu < 0$, there exists fully-faithful functors,
\begin{displaymath}
 \Phi_d : \dbcoh{ Y} \to \dbcoh{ X}
\end{displaymath}
 and
\begin{displaymath}
 \Upsilon_+: \dbcoh{ W } \to \dbcoh{X},
\end{displaymath}
 and a semi-orthogonal decomposition, with respect to $\Phi_d, \Upsilon_+$,
\begin{displaymath}
 \dbcoh{ X } = \langle \dbcoh{W}(-d), \ldots, \dbcoh{W}(\mu-d+1), \dbcoh{Y} \rangle,
\end{displaymath}
 with $\dbcoh{W}(l)$ the image of $\Upsilon_+$ twisted by a character of $\lambda$-weight $l$. 
\end{enumerate}
\end{theorem}

Theorem \ref{theorem: SOD of derived categories for toric variations} is the main tool in producing an exceptional collection for any projective toric DM stack, $X$.  See Remark~\ref{rem: mirror to d} for a discussion of the meaning of the integer $d$ in the mirror.

 The exceptional collection will depend on a choice of presentation,
\begin{displaymath}
 X \cong \mathbb{A}^n \modmod{\chi} G,
\end{displaymath}
of $X$ as a GIT quotient and a path, $\gamma$, in the GIT fan for the action of $G$ on $\mathbb{A}^n$. The path, $\gamma$, determines a run of the toric Mori program for $X$. 

\begin{theorem} \label{thm: SOD on B-side}
 Let $X' = \mathbb{A}^n \modmod{\chi} G$ be a smooth nef-Fano projective toric DM stack, $\gamma$ be a run of the toric Mori program for $X'$, and $X$ be a phase along $\gamma$.  Let $(W_1, \mu_1), \ldots, (W_s, \mu_s)$ be the wall data associated to running $\gamma$ starting at $X$.  There exists fully-faithful functors,
 \begin{displaymath}
  \Upsilon^j_i: \dbcoh{ W_i } \to \dbcoh{X}, \ 1 \leq j \leq \mu_i
 \end{displaymath}
 and a semi-orthogonal decomposition,
 \[
 \dbcoh{X} = \langle \Upsilon^1_s, \ldots, \Upsilon^{\mu_s}_s, \ldots,\Upsilon^1_1, \ldots, \Upsilon^{\mu_1}_1 \rangle.
 \]
\end{theorem}

\begin{proof}
Let $C_0$ be the chamber of the GKZ fan on $X$ corresponding to the nef cone of $X$ and $C_{-K}$  be the chamber of the GKZ fan corresponding to the nef cone of $X'$.

Moving along $\gamma$ from $C_{-K}$ to $C_{X}$, we pass through a finite number of chambers, $C_{-s}, \ldots, C_0$, separated by walls, $\tau_{-s}, \ldots, \tau_{-1}$, with wall data, $(W_1, \mu_1), \ldots, (W_s, \mu_s)$,  where $C_{-s} := C_{-K}$ and $C_0 := C_{X}$.   Let $X_i$ be the GIT quotient corresponding to the chamber $C_i$ so that $X_{-s} = X'$ and $X_0 = X$. 

Since we are on a run of the toric Mori program, $\mu_i \leq 0$ for all $i$.  Therefore by Theorem~\ref{theorem: SOD of derived categories for toric variations}, there is a semi-orthogonal decomposition,
\[
\dbcoh{X_{i-1}} =  \langle \dbcoh{W_{i-1}}, \ldots, \dbcoh{W_{i-1}}(\mu_{i-1}-1), \dbcoh{X_i} \rangle,
\]
where $W_{i-1}$ is the projective toric DM stack appearing in the theorem. Compiling these semi-orthogonal decompositions, we obtain the final desired semi-orthogonal decomposition.
\end{proof}

Let $T$ be an edge-marked tree satisfying the following conditions:
\begin{itemize}
\item for each vertex there are at most 2 outgoing edges,
\item if a vertex has 2 outgoing edges they are marked by $+$ and $-$,
\item if a vertex has 1 outgoing edge it is marked by $+$.
\end{itemize}

Let $X$ be a toric $DM$ stack.  Given a point $p \in \mathcal F_{\op{GKZ}}(X)$ define $m(p)$ to be the minimum over the dimension of all cones containing $p$.

\begin{definition} \label{def: mori tree}
With notation as above, a \textbf{Mori tree} for $X$ is an injective continuous map,
\[
\alpha: T \to \mathcal F_{\op{GKZ}}(X)
\] 
satisfying:
\begin{itemize}
\item for an edge $E$ marked by $-$ with incoming vertex $v$, $m(p)$ is constant along $\alpha(E \cup v)$,
\item for an edge $E$ marked by $+$ with incoming vertex $v$ and any $p \in E$, $m(\alpha(p)) = m(\alpha(v))+1$,
\item for any edge $E$ and outgoing vertex $v$ and any $p \in E$, $m(\alpha(p)) = m(\alpha(v))+1$
\item for any root $r$ of $T$, $\alpha(r) = 0$.
\end{itemize}
\end{definition}

\begin{corollary}
 To any Mori tree $\alpha$ and any GIT presentation of $X \cong \mathbb{A}^n \modmod{\chi} G$, one can associate an exceptional collection on $\dbcoh{X}$.
\end{corollary}

\begin{proof}
 One can choose a character, $\chi'$, in the interior of a chamber of the GIT fan for $\mathbb{A}^n \modmod{\chi} G$ such that $X '= \mathbb{A}^n \modmod{\chi'} G$ is a nef-Fano projective toric DM stack. We can apply Theorem \ref{thm: SOD on B-side} to get a semi-orthogonal decomposition
 \begin{displaymath}
  \dbcoh{X} = \langle \overbrace{\dbcoh{W_s},\ldots,\dbcoh{W_s}}^{-\mu_s}, \ldots, \overbrace{\dbcoh{W_1},\ldots,\dbcoh{W_1}}^{-\mu_1} \rangle 
 \end{displaymath}
 where $(W_i,\mu_i)$ are the wall data coming from the line in $\alpha$ given by the $+$ edges. It now suffices to show that each $W_i$ admits a full exceptional collection. From the vertex in $\alpha$ corresponding to the wall for $(W_i,\mu_i)$, we can take the straight line path given by first taking the $-$ edge and then all $+$ edges. Again applying Theorem \ref{thm: SOD on B-side}, we have 
 \begin{displaymath}
  \dbcoh{W_i} = \langle \overbrace{\dbcoh{W'_{s'}},\ldots,\dbcoh{W_{s'}}}^{-\mu_{s'}'}, \ldots, \overbrace{\dbcoh{W'_1},\ldots,\dbcoh{W'_1}}^{-\mu_1'} \rangle.
 \end{displaymath}
 where $(W'_i,\mu_i')$ is the wall data associated to this path.
 
 We continuing branching off at $-$ edges to refine the semi-orthogonal decomposition. The process terminates as the dimension of projective toric DM stack strictly decreases each time we branch off using a $-$ edge. The final components are derived categories of zero-dimensional projective toric DM stacks.  Such stacks admit full exceptional collections corresponding to all irreducible representations of a finite abelian group.
\end{proof}

\begin{remark}
 Although the exceptional collection depends on a GIT presentation of $X = \mathbb{A}^n \modmod{\chi} G$, there is a natural choice: take $\mathbb{A}^n$ to be the spectrum of the Cox ring of $X$.  However, one should note that in the proof above, the GIT fan for $W$ in the wall data $(W, \mu)$ is interpreted as the restriction of the GIT fan to the wall.  With this in mind, the Cox construction for $W$ as a GIT quotient may have additional torus factors even if $X$ does not. 
\end{remark}

\section{The $A$-side} \label{sec: A side}

\subsection{Degenerations of toric Landau-Ginzburg models}

Let $M$ be as in \eqref{eqn: characters exact sequence}.  The $A$-model mirror to a toric nef-Fano stack $X$ is an LG model,  i.e., a function
\begin{equation} w : M \otimes \C^*  \to \mathbb{C} . \end{equation}
The function $w$ is called the \textbf{superpotential} and can be obtained as a Laurent polynomial in the following manner
(\cite{HoriVafa}).  

Write $N = M^\vee = \textnormal{Hom} (M , \mathbb{Z})$ for the dual of $M$, and consider the dual map from \eqref{eqn: characters exact sequence}
\begin{equation} \delta^\vee : (\mathbb{Z}^n)^\vee \to N . \end{equation}
Standard toric descriptions say that $X$ has a stacky fan in $N$ with one-dimensional cones spanned by $\Sigma (1) = \{\alpha_i := \delta^\vee (e_i^\vee) : 1\leq i \leq
n\}$. 
Let $A = \Sigma (1) \cup \{0\}$ and define $z^\alpha$ to be the character associated to $\alpha \in N$ via the
natural isomorphism $N \cong \textnormal{Hom} (M \otimes \mathbb{G}_m , \mathbb{G}_m )$. Then the superpotential is defined
as 
\begin{equation}
 w = \sum_{\alpha \in A} c_\alpha z^{\alpha} ,\label{eqn: potential} 
\end{equation}
where the coefficients are chosen generically so that all critical points are Morse and have unique
critical values. 

\begin{remark} As $w$ is defined only using the exact sequence \eqref{eqn: characters exact sequence}, toric stacks yielding an identical exact sequence have the same mirror superpotentials. This ambiguity is corrected for upon compactification, as below. As discussed in the introduction, we can only expect that $w$ gives rise to an exact mirror when $X$ is assumed Fano. 
\end{remark}

\begin{remark}In practice, the set $A$ may be a proper subset of the lattice points of $Q$. For later constructions it is important that we record this data throughout. Geometrically this data can be encoded in several ways. \cite{GKZ} introduced a possibly non-normal toric variety $X_A$ whose normalization is the usual projective toric variety $X_Q$ associated to $Q$. Alternately, $A$ determines a linear subsystem of the ample line bundle $\mathcal{O}_Q(1)$ on $X_Q$.  In \cite{DKK12}, the action of the torus on this linear system was used to define a stack $X_{(Q,A)}$. By default this is what is meant by this notation throughout this paper. However, this subtlety will only appear tacitly from here on. 
\end{remark}

 One observes that a choice of generic coefficients of $w$ will not alter the Fukaya-Seidel category
$\fs{w}$ as different choices produce equivalent symplectic Lefschetz fibrations (this follows from an
application of Moser's trick). This is mirror to the situation on the $B$-side where the choice of a K\"ahler structure
does not effect the derived category of coherent sheaves. However, if one considers mirror symmetry with the roles of the $A$- and $B$-models interchanged (e.g. the Fukaya category of $X$ and the singularity category of $w$), the coefficients are important and are prescribed by the K\"ahler structure on $X$ via the mirror map (\cite{ChoOh, HoriVafa}). 

In order to degenerate LG models $w:(\mathbb{C}^\ast )^n\rightarrow\mathbb{C}$, we first compactify the base to $\P^1$ and then partially compactify the total space. To this end, define $X^\vee$ to be the toric stack associated to the convex hull $Q$ of $A$; this is obtained by considering the normal fan to $Q$ as a stacky fan with all generators primitive. The toric stack $X^\vee$ is equipped with the ample line bundle $\mathcal{O}_Q (1)$ which has a basis of sections given by the lattice points $N \cap Q$. As $A \subset N \cap Q$, we may consider the linear system $V_A \subset H^0 (X^\vee , \mathcal{O}_Q (1))$ associated to $A$, which is simply the span of the monomials $\{z^\alpha : \alpha \in A\}$. We write $V^{vf}_A$ for the torus of all such sections that have non-vanishing coefficients for all $\{z^\alpha \}_{\alpha \in A}$, modulo the diagonal $(1,\ldots ,1)$.  We call  a section in $V^{vf}_A$  \textbf{very
full}. Note that by our choice of $A$ we always have a section $s_0=z^0$ corresponding to the lattice point $\{ 0\}$. As $X^\vee$ arises from a stacky fan in $M$, we have that the open dense orbit in $X^\vee$ is $M \otimes
\mathbb{G}_m$. The following proposition follows immediately from the definitions.

\begin{proposition} \label{prop: mirror pencil} With notation as above, there exists a unique $s \in V^{vf}_A$ such that the diagram 
\begin{equation}
\begin{CD} M \otimes \mathbb{G}_m @>>> X^\vee \backslash B \\
@VwVV @V{[s : s_0]}VV \\
\mathbb{C} @>>> \mathbb{P}^1
\end{CD}
\end{equation}
commutes, where $B =\textnormal{Zero} (s_0 ) \cap \textnormal{Zero} (s)$ is the base locus of the pencil $[s:s_0]$. We call $[s : s_0]$ a {\bf mirror pencil} for $X$.
\end{proposition}

By this proposition, we can equivalently consider the toric mirror superpotential as a one parameter family of 
hypersurfaces in $X^\vee$ moving in the linear system $V_A$. One advantage of this perspective is that the moduli theory of such hypersurfaces and their compactifications  has been thoroughly studied (e.g \cite{Alexeev, DKK12, GKZ}), and is combinatorial in nature. Note that the torus $T_M =M\otimes\mathbb{C}^\ast$ acts on $V_A^{vf}$ by coordinate scaling. The following theorem, adapted from Theorem 2.18 of \cite{DKK12}, describes the structure of this toric moduli space. 

\sidenote{{\color{red} Too much confusion between toric stack definition and usage here.}}
\begin{theorem}[\cite{DKK12}] \label{thm:moduli}
There exists a toric moduli stack $\mathcal{X}_{\Sigma (A )}$ for very full toric
hypersurfaces $H$ and their toric degenerations. In particular, $\mathcal{X}_{\Sigma (A )}$ is a compactification of $V^{vf}_A/T_M$. The universal hypersurface $\mathcal{H}$ is contained in a toric stack
$\mathcal{X}_{\Theta (A)}$ which admits a flat map 
\begin{equation*} \pi : \mathcal{X}_{\Theta (A)} \to \mathcal{X}_{\Sigma (A )}. \end{equation*}
The polytope $\Sigma (A)$ yielding the stack $\mathcal{X}_{\Sigma (A )}$ is the secondary polytope  of $(Q,A)$. The polytope $\Theta (A)$ yielding the stack $\mathcal{X}_{\Theta (A)}$ is a Minkowski sum of $\Sigma (A)$ with a simplex. \end{theorem}

The normal fan of the polytope ${\Theta (A)}$ was introduced by Lafforgue \cite{Laff} for matroid polytopes. This construction is also implicit in work of Kapranov \cite{KapChow}, and the fan appears explicitly in Hacking's subsequent work \cite{HackingModuli}. The secondary polytope was introduced and thoroughly investigated in \cite{GKZ}, and a more general theory called fiber polytopes was introduced by Billera and Sturmfels \cite{BS}. This includes as a special case a polytope called the monotone path polytope which we will encounter later. We will not review the full theory of secondary or Lafforgue polytopes here, but will review the most pertinent concepts below.

The most relevant part of this structure for this paper is the 1-skeleton of the secondary polytope. The vertices correspond bijectively to {\bf convex triangulations} $T$ of $(Q,A)$. These are simply cellular decompositions of the polytope $Q$ into lattice simplices whose vertices are respectively contained in the set $A$, and subject to the condition that they arise from the projection of the linearity domains of a piecewise affine convex function. Given such a $T$, the corresponding vertex $\phi_T\in\mathbb{R}^A$ of $\Sigma (A)$ is given simply by 
\begin{equation}
\phi_T^\alpha  = \sum_{\alpha\in\sigma}\op{vol}(\sigma )
\label{eq: vertices}
\end{equation}
 where $\phi_T^\alpha$ denotes the corresponding coordinate in $\mathbb{R}^A$, and the sum is over the normalized lattice volumes of all simplices containing the lattice point $\alpha\in A$. Of particular interest to us are the edges of ${\Sigma (A)}$ which  correspond bijectively to extended circuits of $(Q,A)$. 

\begin{definition}\label{def: circuit} An (affine) {\bf circuit} in a lattice $L$ is a set $C\subset L$ which is affinely dependent over $L$, but such that every proper subset is affinely independent. If $L$ has rank $d$, an {\bf extended circuit} is any subset $\tilde{C}\subset L$ such that $|\tilde{C}| = d+2$ and which affinely spans $L$. 
\end{definition}

It is easy to check that any extended circuit contains a unique circuit. 
We may write an extended circuit explicitly by writing the affine relation $$\sum_{\alpha\in A}c_\alpha\alpha =0$$ where $\sum_{\alpha\in A}a_\alpha =0$. Let $A^+= \{\alpha\,|\, a_\alpha >0\} $, $A^-= \{\alpha\,|\, a_\alpha <0\} $, and $A^\circ=\{\alpha\,|\, a_\alpha =0\} $. The triple $(|A^+|,|A^-|;|A^\circ |)$ is an invariant of an extended circuit called the {\bf signature}, and is well-defined up to a switch of the first two entries. If $|A^\circ| = 0$, we will write $(|A^+|, |A^-|)$ for the signature and refer to the extended circuit as non-degenerate. The relationship between (extended) circuits and triangulations is that every extended circuit determines two (convex) triangulations $T^\pm$ of the polytope $Q=\op{conv}(A)$ by taking $$T^\pm = \{\op{conv}(A\setminus\{\alpha_i\}\, |\, \alpha_i\in A^\pm\}$$ as the decomposition of $Q$ into simplices with vertices contained in $A$. 


\begin{remark}A similar theory can be developed with simplicial sub-fans supported on a set of rays playing the role of triangulations of a marked polytope. Here a (linear) circuit will correspond to sets of primitive generators which are linearly dependent but such that every proper subset is linearly independent, and these linear circuits will correspond bijectively to the walls of the GKZ fan. Indeed, the entire theory of the GKZ fan of a toric variety can be developed in this language;  this was the viewpoint taken in the original work of Reid \cite{Reid83} and Oda-Park \cite{odapark}, and is discussed in detail in \cite[Section 15.3]{CLS}. 
\end{remark}

When $X$ is a toric Fano variety, the relationship between the GKZ fan and the secondary polytope is particularly clear. Let $\Delta$ be a reflexive polytope defining $X$, $\Delta^\circ$ to be the polar dual, and define the set $A= \op{vert}(\Delta^\circ )\cup\{ 0\}$. The following result is well-known, see e.g. \cite[Section 7]{Coxrecent}.

\begin{proposition} With notation as above, the enlarged GKZ fan from Proposition \ref{prop: enlarged} equals the normal fan of the secondary polytope $\Sigma (A)$. In particular, the GKZ fan of $X$ is a full-dimensional subfan of the normal fan of $\Sigma (A)$. 
\end{proposition}This proposition will be used at least implicitly many times in order to identify walls in the GKZ fan (B-side) with circuits on the mirror ($A$-side). 

Let us now interpret the above combinatorial constructions geometrically. Following \cite{Alexeev}, one associates to $T$ (or more generally any face of $\Sigma (A)$) a {\bf stable toric variety} which is a possibly reducible variety whose components are the toric stacks corresponding to the respective cell of the subdivision of $Q$, and which are glued together along boundary strata according to the combinatorics of the subdivision. A fiber of $\mathcal{H}\rightarrow\mathcal{X}_{\Sigma (A )}$ is a {\bf degenerated hypersurface}; namely a collection of reducible hypersurfaces in each of the components of a stable toric variety with the respective gluing. The structure of a degenerated hypersurface is closely related to the corresponding tropical variety; we review this in \ref{subsubsec: tropical}.

A specific consequence of the above theorem is that one can realize a mirror pencil from \ref{prop: mirror pencil} as the pullback along a map $\iota : \mathbb{P}^1 \to \mathcal{X}_{\Sigma (A)}$ as in the diagram below
\begin{equation}
 \begin{CD} M \otimes \mathbb{G}_m @>i>> X^\circ \backslash B @>\iota^*>> \mathcal{H} \\
 @V\mathbf{w}VV @V{[s : s_0]}VV @V{\pi |_\mathcal{H}}VV \\
 \mathbb{C} @>>> \mathbb{P}^1 @>\iota>> \mathcal{X}_{\Sigma (A)}
 \end{CD}
\end{equation}

We use this description to construct the moduli space of LG models; that is, the moduli space of mirror pencils. 
\begin{remark}The theory of maps from a sphere to a symplectic variety is quite rich and, were $\mathcal{X}_{\Sigma (A)}$ 
a smooth variety (or DM stack), we could study the structure of the (twisted) stable map compactification
$\mathcal{M}_{0} (\mathcal{X}_{\Sigma (A)}, \iota (\mathbb{P}^1))$ in order to give a useful description of the space of all such mirror pencils. Unfortunately, it is quite rare that $\mathcal{X}_{\Sigma (A)}$ is smooth even as a stack. The approach we give instead better reflects the underlying combinatorics.\end{remark}
\sidenote{{\color{red} References for stable maps maybe?}}

Let  \begin{equation} \rho : V^{vf}_A \to \mathcal{X}_{\Sigma (A)} \end{equation} be the map which is the composition of the quotient $V^{vf}_A\rightarrow V^{vf}_A/T_M$ with the inclusion of $V^{vf}_A/T_M$ onto the dense torus of $\mathcal{X}_{\Sigma (A)}$.  Now, let $K = \mathbb{G}_m$ be the one parameter subgroup of the open torus of $\mathcal{X}_{\Sigma (A)}$ which acts via \begin{equation}
 \lambda \cdot \rho \left( \sum_{\alpha \in A} c_\alpha s_\alpha \right) = \rho \left( \lambda c_0 s_0 + \sum_{0
\ne \alpha \in A}  c_\alpha s_\alpha \right) .\end{equation}
By varying $s=\sum_{\alpha\in A}c_\alpha s_\alpha$ and comparing with the definition of a mirror pencil as in Proposition~\ref{prop: mirror pencil}, we obtain the following useful characterization. 

\begin{proposition}
 Any mirror pencil for $X$ is represented by the closure of a $K$-orbit in $\mathcal{X}_{\Sigma (A)}$.
\end{proposition}

Being a one-parameter subgroup, $K$ corresponds equivalently to a cocharacter; let us now describe it explicitly. Consider the exact sequence $$0\to M\oplus\mathbb{Z}\overset{\tilde{\delta}}{\to}\mathbb{Z}^A\overset{\tilde{\gamma}}{\to} \widehat{G}\to 0$$ where $\tilde{\delta}(m,\lambda ) = \sum_{\alpha\in A}(\langle m,\alpha\rangle+\lambda )$, and $\widehat{G}$ is defined to be the cokernel. This exact sequence is just the $A$-side analogue of \eqref{eqn: characters exact sequence}.
 Recall that in our situation $A=\op{vert}(\Delta )\cup\{0\}$, and let $A' = \op{vert}(\Delta )$, which we may identity with the primitive generators of the fan defining the original toric Fano.

Comparing with the exact sequence \eqref{eqn: characters exact sequence}, consider the commutative diagram 
\begin{equation*}\begin{CD} 0 @>>> M @>\delta>> \mathbb{Z}^{A'} @>\gamma>> \widehat{G} @>>> 0 \\ @. @V{i_1}VV @V{i_2}VV @V=VV @. \\0 @>>> M \oplus \mathbb{Z} @>{\tilde{\delta}}>> \mathbb{Z}^{A} @>{\tilde{\gamma}}>> \widehat{G} @>>> \end{CD}\end{equation*}where $i_1$ is inclusion in the first factor, $i_2$ is the inclusion $\mathbb{Z}^{A'}\subset\mathbb{Z}$, and\begin{equation} \tilde{\delta} (m , k) = i_2(\delta (m)) + k \sum_{\alpha \in A} e_\alpha . \end{equation} 
The cocharacter of $K$ acting on $\mathbb{Z}^A \otimes \mathbb{G}_m$ (before applying $\rho$) is clearly $e_0$. This descends via $\tilde{\gamma}$ to the cocharacter $\chi = \tilde{\gamma} (e_0)$. 

\sidenote{{\color{red} $\chi$ cannot be a chararacter and a cocharacter}}
\begin{remark} Note that as an element of $\widehat{G}$, the cocharacter $\chi$ corresponding to the subgroup $K$ coincides with the anticanonical divisor of the original variety $X$. \end{remark}

It is thus natural to define the space of mirror pencils to be a quotient of the stack $\mathcal{X}_{\Sigma (A)}$ by $K$. This requires the correct notion of quotient. The most structured approach is to use the exact sequences above, along with the generalized notion of a toric stack developed in \cite{GS}, to take this quotient at the level of generalized stacky fans. If we only regard the toric pencils as equivariant cycle classes determined by $\iota (\mathbb{P}^1)$, we may take a more classical approach.

\begin{definition} The space of LG mirrors $\langinz$ is the toric Chow quotient of $X_{\Sigma (A)}$ by $K$.
\end{definition}

For the definition and properties of this quotient, see \cite{KSZ}. Note that in general this will be a non-normal toric variety. It was observed in \cite{CS12} that this space admits an equivalent description in terms of log stable maps. It is important to note that the polytope corresponding to $\langinz$ was discussed as a purely combinatorial object in \cite{BS}, where it was dubbed the {\bf monotone path polytope}. More precisely, this polytope is the monotone path polytope of $\Sigma (A)$ relative to the hyperplane. The most important feature of this polytope is that its vertices correspond to edge paths on the secondary polytope. Intuitively, the cocharacter $\chi = \tilde{\gamma} (e_0)$ determines a linear functional on the space $\widehat{G}_\mathbb{R}$ which contains the secondary polytope, and the edge paths are those which move ``downward" with respect to this linear functional.

Taking a page from mirror symmetry of Calabi-Yau varieties, we consider Homological Mirror Symmetry of the original
nef-Fano stack as taking place near the toric fixed points in $\langinz$; that is, near the large complex structure limit. 

\begin{definition}
 A torus fixed point of $\langinz$ is called a maximally degenerated LG mirror.
\end{definition}

The following theorem relates these degenerated LG models to their $B$-side mirrors, and makes precise the notion that the edge paths specified by vertices of the monotone path polytope are mirror to runs of the Mori program. A complete proof of this can be found in
\cite{DKK12}, 

\begin{theorem}[\cite{DKK12}]
Maximal degenerated LG mirrors are in bijection with straight line run of the toric Mori program on $X$. 
\end{theorem}

Let us now describe the geometry of maximally degenerated LG mirrors in more detail. As mentioned above, they in particular correspond to a collection of edge paths on the secondary polytope, and thus a sequence of circuits. Geometrically, to every maximal degeneration $\psi\in\langinz$ there is thus an associated 1-cycle \begin{equation} Z_\psi = \cup_{i = 1}^m a_i Z_{\psi , i}
\end{equation} on $\mathcal{X}_{\Sigma (A)}$. Equivalently, this is a degeneration of the generic mirror pencil \ref{eqn: potential} into a (weighted) union of one dimensional toric boundary strata. The coefficients of the cycle can be seen by taking $\{\psi (t) \}_{t \in [0,1]} \subset \langinz$ to be a one dimensional
family of LG pencils with $\psi (0) = \psi$ a maximal degeneration, and observing a bubbling phenomena for the corresponding maps
$\psi (t) : \mathbb{P}^1 \to \mathcal{X}_{\Sigma (A)}$ into $a_i$-ramified maps 
\begin{equation}
 \psi_i (0) : \mathbb{P}^1 \to Z_{\psi , i} .
\end{equation}
By the description of the universal hypersurface $\mathcal{H}$ over $\op{X}_{\Sigma (A)}$, we thus see that $\psi$ defines a chain of (stacky) $\mathbb{P}^1$'s, which is equipped with a family of degenerated hypersurfaces. Over the nodes of this chain and over the two torus fixed points at the end of the chain the hypersurface is typically reducible and has irreducible components consisting of (stacky) pairs of pants glued according to the corresponding triangulation. The generic fiber over a components of the chain will likewise be determined by the circuit determined by the two triangulations of the respective endpoints. The advantage of these degenerate pencils is that they are essentially entirely combinatorial. In practice, we will need Lefschetz pencils which approximate these degenerate pencils. This can be done by considering a point in $\langinz$ sufficiently close to the corresponding fixed point. We will call such a pencil a {\bf regeneration} of $\psi$. In practice, there are 
subtleties due to the non-trivial stack structure of $\mathcal{X}_{\Sigma (A)}$, see Section 4.3 of \cite{DKK12} for details. 

There are prescribed asymptotics in $t$ for annuli $B_i$ in the domain of $\psi (t)$ for which $\psi (t)|_{B_i}$
converges to the components $\psi_i (t)$ (\cite{DKK12}). These asymptotics separate the critical values of $\psi (t)$
into subsets $\{S_1, \ldots, S_m\}$. Furthermore, as $\psi_i (t)$ converges to an $a_i$-fold cover of $Z_i$, each subset $S_i$ of critical values further partitions into $a_i$ groups $\{S_{i1}, \ldots, S_{i a_i} \}$ corresponding to the branches of the cover. 
To summarize, given a maximal degeneration  $\psi (0) = \psi\in\langinz$, and a one-parameter family of LG-models approaching it, we obtain a decomposition of the plane into successive annuli such that:
\begin{itemize}
\item There is one annulus for each circuit appearing in the edge path on $\Sigma (A)$ determined by $\psi$

\item If an edge appears with multiplicity, the corresponding annulus is divided into the corresponding number of radial regions.

\end{itemize}

The number of critical points in each such region then depends on the further geometry of the corresponding circuit. In particular, it is 1 unless the circuit is what is known as a degenerate circuit. To compute the multiplicities $a_i$, let $\textnormal{Stab} (Z_i)$ be the subtorus in $\widehat{G} \otimes \C^*$ that stabilizes $Z_i$. Then $a_i$ equals the order of the group $\ker (\eta_i )$ where
\begin{equation} \eta_i : K \to \widehat{G} / \textnormal{Stab}(Z_i ).  \end{equation}
The computation of this order amounts to evaluating the cocharacter of $K$ with a generator for the affine edge $\{T_i, T_{i + 1}\}$ in the secondary polytope $\Sigma (A)$ corresponding to $Z_i$. In particular, if $\varphi_{T_i}$ and $\varphi_{T_{i + 1}}$ are the vertices on the edge $\Sigma (A)$ regarded as a polytope in $\Z^A$ and their difference $\varphi_{T_{i + 1}} - \varphi_{T_i}$ is not a nontrivial multiple of any lattice point, then 
\begin{equation}
\label{eq:mult form} a_i = \left< e_0 , \varphi_{T_{i + 1}} - \varphi_{T_i} \right>
\end{equation}
The total collection of the decomposition into annuli and the distinguished basis is called a {\bf radar screen}. We will see an explicit example in Figure~\ref{fig:tropical Morse mirror}.

\begin{remark}
For a run of the Mori program, let $(W_i, \mu_i)$ be the wall data associated to an affine edge $\{T_i, T_{i + 1}\}$. It is shown in \cite{DKK12} that
\[
a_i = - \mu_i\cdot \text{dim K}_0(W_i).
\]
When the corresponding extended circuit is non-degenerate $a_i = - \mu_i$.  Notice that there's a minus sign due to an altercation among sign conventions in \cite{DKK12} and \cite{BFK12}.  In addition, the locus $W$ is determined by the locus where the fan is changed nontrivially i.e. the support of the circuit (see \cite[Section 15.3]{CLS}). 
\label{rem: matching up}
\end{remark}

We may use this result to write down a conjectural mirror
for the semi-orthogonal decomposition given in Theorem \ref{thm: SOD on B-side}. Strictly speaking, if $X'$ is a nef-Fano and $X$ is a phase, the theorem below gives the conjectural mirror semi-orthogonal decomposition of $\dbcoh{X}$, of which a subset of components specified by omitting the wall contributions from $X$ to $X'$ yields $\dbcoh{X'}$.

\begin{theorem}[\cite{DKK12}]\label{thm: fs decomposition} For every maximal degenerated LG mirror $\psi$ of a family $\psi (t)$ there is a
semi-orthogonal decomposition 
\[
\langle \overbrace{ \fs{\psi_1}, \ldots , \fs{\psi_1 }}^{-\mu_1},
\ldots , \overbrace{\fs{\psi_s } , \ldots, \fs{\psi_s }}^{-\mu_s} \rangle.
\]
of the Fukaya-Seidel category $\fs{\psi  (t)}$ for $t \approx 0$. Here the $\psi_i$ are regenerations of the singular pencils $Z_i$. \end{theorem}

\begin{remark} \label{rem: mirror to d}
The integer $d$ from Theorem~\ref{theorem: SOD of derived categories for toric variations} is mirror to a monodromy transformation around a degenerate fiber of $Z_i$ (i.e. a fiber over a torus fixed point of the base stacky $\mathbb{P}^1$), where $Z_i$ is the pencil corresponding to the circuit which is mirror to the corresponding wall crossing. In terms of the radar screen, this amounts to a rotation of the corresponding annulus and hence yields a different semi-orthogonal decomposition of the Fukaya-Seidel category $\fs{\psi  (t)}$ for $t \approx 0$.
\end{remark}

\section{The Mori program meets Homological Mirror Symmetry, torically} \label{sec: program}

\subsection{General method} \label{sec: general method}

We begin this section by recalling our central conjecture \ref{conj:main conjecture}. Following the notation of the previous sections, we let $X$ be a projective toric DM stack, $X^\prime$ be a projective toric nef-Fano DM stack with a straight line
run for $X$ containing $X$ as a phase. We take $\psi = \psi (0)$ to be the degenerated LG mirror corresponding to the Mori program run and $\{\psi (t) \}$ a degenerating family. 
 \begin{conjecture} \label{conj:main conjecture} Given any maximal degeneration $\psi$ of an LG mirror pencil, let 
\begin{eqnarray*} \op{Fuk}^{\rightharpoonup} (\psi (t) ) & = & \langle \mathcal{T}_1, \ldots, \mathcal{T}_r \rangle , \\ \dbcoh{X}
& = & \langle \mathcal{S}_1, \ldots, \mathcal{S}_r \rangle
\end{eqnarray*} be the semi-orthogonal decompositions associated to $\psi$ and its run of the toric Mori program.

Then there exists an equivalence of triangulated categories
\begin{equation*} \Phi_\psi : \op{Fuk}^{\rightharpoonup} (\psi (t)) \to \dbcoh{X} \end{equation*}
which restricts to equivalences $\Psi_\psi : \mathcal{T}_i \to \mathcal{S}_i$ for all $1 \leq i \leq r$.
\end{conjecture} 

In addition to being a statement of Homological Mirror Symmetry for toric varieties,  the conjecture is formulated to suggest a rich correspondence between decompositions of the categories involved, the geometry of the space of LG models, and birational geometry.

Before outlining a strategy to prove this conjecture, we cite some evidence in its favor. It was proven in \cite{DKK12} that the conjecture holds at the level of $K$-theory ranks. 

\begin{proposition}[\cite{DKK12},5.17] 
In the decompositions of Conjecture \ref{conj:main conjecture} 
\begin{equation}
 \textnormal{rank} (K_0 (\mathcal{T}_i )) = \textnormal{rank} (K_0 (\mathcal{S}_i ))
\end{equation}
 for all $1 \leq i \leq r$.
\end{proposition}

\sidenote{{\color{red} Bullet point steps to proving HMS} We can throw this into the next version}

There are two hurdles that must be overcome in order to prove Conjecture~\ref{conj:main conjecture}. The first is
to show that 
\begin{equation}
 \mathcal{T}_i \simeq \mathcal{S}_i
\end{equation}
for any circuit in $\Sigma (A)$, or equivalently, for any wall crossing in the GKZ-fan coming from a straight line run.
The second hurdle is showing that the $A_\infty$- or dg-bimodule of morphisms from $\langle  \mathcal{T}_1,
\ldots, \mathcal{T}_k \rangle$ to $\mathcal{T}_{k + 1}$ is quasi-equivalent to that from $\langle  \mathcal{S}_1,
\ldots, \mathcal{S}_k \rangle$ to $\mathcal{S}_{k + 1}$. The full equivalence would then follow from an easy induction. For both the $A$ and $B$ side of mirror symmetry, a detailed study of the dg- or $A_\infty$- endomorphism algebra associated to exceptional collections in the respective category must be made.

\subsubsection{Calculating the $A_\infty$-algebra corresponding to a run of the toric Mori program on the $A$-side}\label{subsubsec: tropical}

For the $A$-side, we sketch a correspondence between LG models and a type of ``tropical Morse theory''. The main goal of which is to understand the structure of the Lagrangian vanishing cycles which appear in the regenerated LG models, as this is precisely what determines $\fs\psi $ by Theorem \ref{thm: fs decomposition}. The main point is that the location of these vanishing cycles can be largely described by determining their tropicalizations. The idea is simple: recall that for each component in the chain of (stacky) $\mathbb{P}^1$'s given a maximal LG degeneration $\psi$ its two torus fixed points correspond to triangulations, and the entire components corresponds to a circuit. We will show that the tropicalization of the vanishing cycles of a regeneration are determined by the circuit data. 

To this end, we first recall some basic notions from tropical geometry. First, recall that given a polarized toric variety $X_Q$ determined by a lattice polytope $Q$, there is a continuous map $$\op{Mo} : X_Q\rightarrow 
Q$$ called the moment map, such that the fiber over a point in the interior of a codimension $k$ cell of $Q$ is the real torus $(S^{1})^k$. In practice, it will suffice to consider the restriction of $\op{Mo}$ to the dense torus, in which case, identifying the interior of $\Delta$ homeomorphically with $\mathbb{R}^n$, the map is identified with the usual amoeba map $$\op{Log}(t_1,\ldots ,t_n) = (\log |t_1|,\ldots ,\log |t_n|). $$

If $H\subset X_Q$ is a hypersurface, the image $\op{Mo} (H)\subset Q$ is called the {\bf amoeba} of $H$. Tropical geometry studies the structure of $H$ via polyhedral complexes onto which the amoeba deformation retracts.

More precisely, given a convex triangulation $T$ of a polytope $(Q,A)$, we may define a tropical variety $Y_T$ which is the dual polyhedral complex to the triangulation, with maximal cells weighted by the lattice volume of the corresponding dual simplex. To each simplex in the triangulation, we associate a weighted tropical pair of pants which glue together to give the full tropical variety $Y_T$. 

Recall from \eqref{eq: vertices} that $T$ corresponds to a vertex of the secondary polytope $\Sigma (A)$ and thus a unique fixed point $p_T$ of the stack $\mathcal{X}_{\Sigma (A)}$. The relationship between the geometry of a hypersurface and $Y_T$ is most clear when the triangulation is assumed to be {\bf maximal}; that is, when the maximal simplices of $T$ each have normalized lattice volume equal to 1. As before, if $p\in\mathcal{X}_{\Sigma (A)}$, we let $\mathcal{H}_q$ denote the hypersurface in $X_Q$ which is the fiber of $\pi_{\mathcal{H}}$ over $p$. Let $\mathcal{H}^\circ_p$ denote the open part of this hypersurface which is contained in the dense torus of $X_Q$. 

The following proposition follows immediately from the main results of \cite{pants}.
\begin{theorem}
Let $T$ be a maximal triangulation of $(Q,A)$, and let $q(t)$ be a continuous path in $\mathcal{X}_{\Sigma (A)}$ such that $q(0)=p_T$ and $q(t)$ does not meet the toric boundary of $\mathcal{X}_{\Sigma (A)}$ for $0<t<\epsilon$. Then as $t\rightarrow 0$, the amoebas $\op{Mo} (\mathcal{H}^\circ_{q_t})$ converge, in the Gromov-Hausdorf limit, to the tropical variety $Y_T$.
\end{theorem} 


Analyzing the convergence for sufficiently small $t$ produces a continuous map 
\begin{equation}
 \pi^{trop} : \mathcal{H}^\circ_{q_t} \to Y_T
\end{equation}
called a {\bf tropical degeneration} of $\mathcal{H}_{q_t}$. The fibers of this map over the interior of a maximal cell are simply real tori.

We now apply this machinery to the study the fibers of mirror pencils. Recall from Section~\ref{sec: A side} that a maximal degeneration of a mirror pencil arises as the limit of a family $\{\phi (t)\}$ of mirror pencils whose limit in the space $\langinz$ determines a sequence of edge paths $\{T_i, T_{i+1}\}$. Thus, for each $1 \leq i \leq m$ and $t \in [0,1]$ we may define $\{q_{i , t}\}$ on the boundary of the $i$-th annulus $B_i$ of in the $K$-orbit $\phi (t)$ which converges to $p_{T_i}$. 
In particular, the fibers near the maximal degeneration $\phi$ can be viewed as going through a sequence of tropical degenerations
\begin{equation}
 \pi^{trop}_i : \phi (t)^{*} (Z_{q_{i, t}} ) \to Y_{T_i} .
\end{equation}

We now discuss how to use tropical techniques to study not just the fibers over the points $q_{i,t}$, but the vanishing cycles which appear in the regenerated components $\psi_i$ from Theorem \ref{thm: fs decomposition}.  It was noted in \cite[Example 11.5.8]{GKZ} that there is a strong relationship between triangulations related by a circuit and Morse surgery. This motivates the following definition. 

\begin{figure}[ht]
 \includegraphics{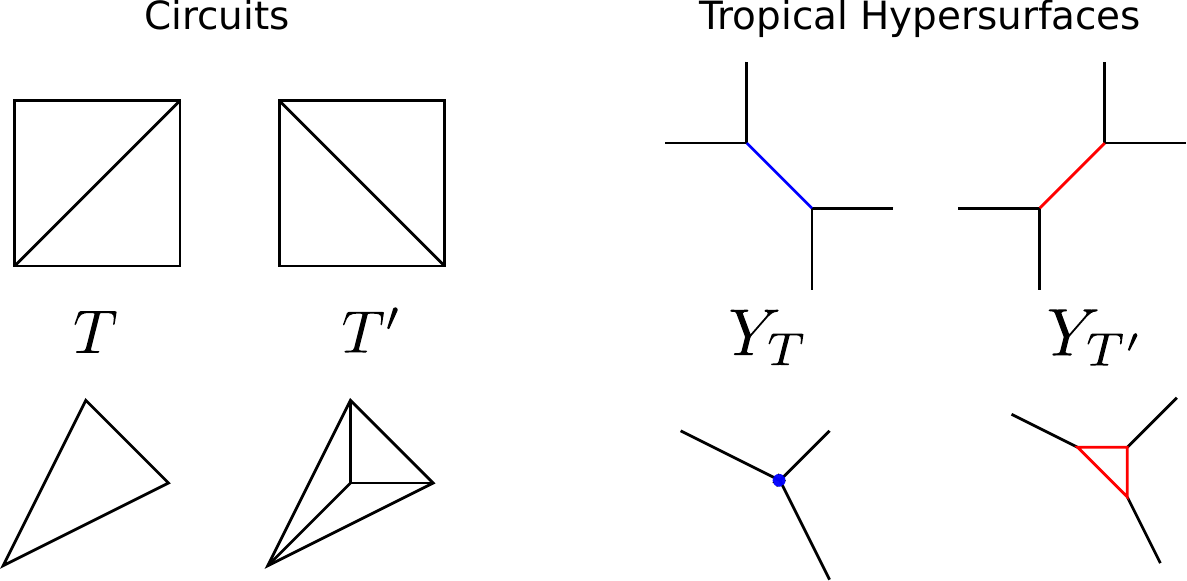}
 \caption{\label{fig:stable trop} Unstable and stable complexes for two dimensional circuits.}
\end{figure}
\begin{definition}
Let $T$ and $T^\prime$ be triangulations related by a circuit, and let $Y_{T}$ and $Y_{T^\prime}$ be the respective tropical varieties. A cell $P \subset Y_{T^\prime}$ is a {\bf stable cell} with respect to the circuit if its dual simplex is contained in $T^\prime$ but not $T$. The {\bf unstable cells} are those in $Y_{T}$ whose dual simplices are contained n $T$ but not $T'$. The collection of all stable cells will be called the {\bf stable complex} of the modification.
\end{definition}

Of course, the definition depends on an ordering of $\{T,T'\}$. In all situations we consider, this is afforded by the sequence $T_0,\ldots , T_ n$ defined by a maximal LG degeneration $\psi$. In Figure~\ref{fig:stable trop} we illustrate the stable cells in
red and the unstable cells in blue for non-degenerate circuits in two dimensions. The main reason for these considerations is to apply the technology of stable and unstable complexes to the study of vanishing cycles which appear in degenerations of LG models. We prove here that the stable complex is indeed the tropicalization of a vanishing cycle. 

\begin{proposition} \label{prop:tropical vanishing}
 Let $\{\psi (t)\}$ be a maximal degeneration of a mirror potential with $\{T_0, \ldots ,T_m \}$ the corresponding sequence of triangulations. If $\gamma$ is a path in $B_i$ from a critical value of $\psi (t)$ to $q_{i,t}$ with vanishing cycle $L$, then the image $\pi^{trop}_i : L \to Y_{T_i}$ converges to the stable complex of $\{T_i, T_{i + 1}\}$ as $t \to 0$.
\end{proposition}

\begin{proof}
Recalling Theorem~\ref{thm:moduli}, we have the map $\pi : \mathcal{X}_{\Theta (A)} \to \mathcal{X}_{\Sigma (A)}$ along
with its restriction $\pi_{\mathcal{H}}$ to the universal hypersurface. On the annulus $B_i$, the maps $\psi (t) :
\mathbb{P}^1 \to \mathcal{X}_{\Sigma (A)}$ converge to an $a_i$-fold cover of the $\mathbb{G}_m$ orbit corresponding to the edge $\{T_{i - 1}, T_i\}$ on the polytope $\Sigma (A)$. Let $S$ be the convex subdivision corresponding to the edge $\{T_{i - 1} , T_i\}$ and note that the polytopes in $S$ consist of the simplices that are
members of both triangulations and the polytope $C$ which contains all simplices of which are being modified.

The pullback via $\psi (0)$ of the hypersurface $\mathcal{H} \subset \mathcal{X}_{\Theta (A)}$ along the component
$Z_i$ decomposes into components
\begin{equation}
 \psi (0)|_{Z_{\psi, i}}^* ( \mathcal{H} ) = \cup_{\sigma \in S} X_{\sigma}
\end{equation}
where $X_\sigma$ is a fiber bundle of a of pants (or a cover thereof) over $\mathbb{P}^1$ for all $\sigma \ne C$ and the
map $f_C :X_C \to \mathbb{P}^1$ is associated to the subpolytope $C \subset Q$ associated to the circuit. It is a consequence of Proposition 5.7 of \cite{DKK12} that the critical points of $\psi (t)|_{B_i}$ converges to the
unique critical point of $f_C$ and is contained in the interior of $X_C$. 

Let $\tilde{Y}_{T_i}$ be the tropical variety associated to the $T_i$ triangulation of the marked polytope $(C , C \cap
A)$. Then, there is an inclusion of the polyhedra in $\tilde{Y}_{T_i}$ to those in $Y_{T_i}$. We observe that at $t =
0$, the tropical image of the vanishing cycle $L$ is contained dual complex $\tilde{Y}_{T_i} \subset Y_{T_i}$, as it is
the same image as that of the hypersurfaces corresponding to $f_C$. One can use a compressing flow to force this image
to be in the compact polyhedra of $\tilde{Y}_{T_i}$. It is an easy exercise to see that for a circuit, these polyhedra give the stable complex of the circuit corresponding to $\{T_i |_C, T_{i + 1}|_C\}$. But it easily follows
from the definition of stable complex that the inclusion of the stable complex in $\tilde{Y}_{T_i}$ of $\{T_i |_C, T_{i
+ 1}|_C\}$ into the full tropical complex $Y_{T_i}$ will be a bijection of polyhedra. This yields the claim.
\end{proof}

\begin{remark} Conjectures of Seidel \cite{SeiSpec} suggest that it may be possible to build a calculus that yields the Floer theory of the vanishing cylces in terms of essentially tropical data.  The above proposition indicates that one can identify the vanishing cycles of an LG mirror with cell complexes in the tropical hypersurface $Y_{T_m}$ near infinity; however, the results in \cite{Abouzaid} strongly suggest that, in general, in order to fully recover the Floer theory of a Lagrangian from its tropicalization, one requires additional combinatorial enhancements. In addition to standard tropical techniques, establishing related results would require the study of Floer theory in the coamoeba setting \cite{FHKV09, Sheridan}. Were such a calculus to be built, the $A_\infty$-endomorphism algebra of the exceptional collection of vanishing cycles would then be read off from the combinatorial data of the tropical vanishing cycles, after which a comparison with the $B$-model could be made. The 
implementation of this program is beyond the scope of the work, but motivates much of our analysis. In particular, we view the analysis of the $A$-side of the main example in Section \ref{subsec: example} as a particular instance of this philosophy. \end{remark}

In the case where the fibers are Riemann surfaces, using matching paths arguments, the actual vanishing cycles can be explicitly identified for the non-degenerate circuits (of which there are only two cases). These cycles are pictured in Figures~\ref{fig:stable trop} and \ref{fig:circuit vanishing cycles} as an illustration of Proposition~\ref{prop:tropical vanishing}.

\begin{figure}
 \includegraphics{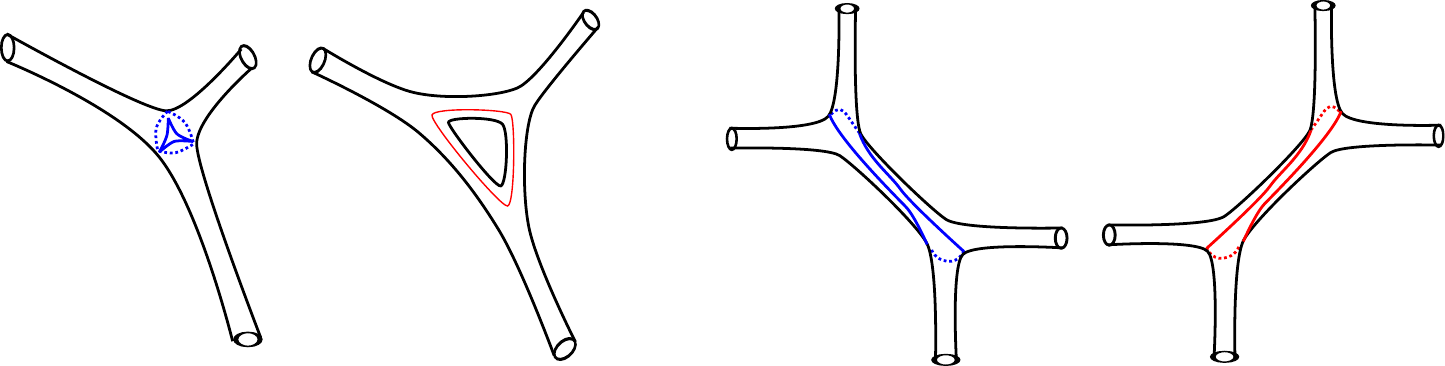}
 \caption{\label{fig:circuit vanishing cycles} Circuit vanishing cycles}
\end{figure}

In the three cases on the right of Figure~\ref{fig:circuit vanishing cycles}, the vanishing cycle is embedded in the hypersurface while in the case on the left,
one sees an immersed vanishing cycle which is a $(\mathbb{Z} / 3 \mathbb{Z})$ quotient of the hexagonal torus. In this last case, the actual Lagrangian lies in the cover. The monodromy around
the fixed points of these circuit fibrations is given by a Dehn twist along each cylinder that
corresponds to an edge in the stable or unstable manifold. In the case of the finite quotient, this is simply the action
of $(\mathbb{Z} / 3 \mathbb{Z} )$ on the cover.  In the next section, we pull these cycles back along the tropical
hypersurfaces associated to a maximal LG mirror degeneration.


\begin{remark}
 Compared to computing Floer cohomology, there are a multitude of general ways for computing morphisms in the derived category of coherent sheaves on a variety. However, the $A$-model analysis described above and the methods and observations of \cite{AbouSei,SeiSpec} suggest an approach to the $B$-side computations of $A_\infty$ structure that might be amenable to mirror symmetry.  Namely, the hypersurface on the $A$-side admits a pair pants decomposition coming from its tropical amoeba. This decomposition is combinatorially dual to the collection of maximal cones in the fan for the corresponding toric variety. If we wish to compute morphisms on the $A$-side, following \cite{SeiSpec}, we should be able to compute the Floer cohomology of the restrictions of the Lagrangians to each pair of pants and glue them together to get the global Floer cohomology. Mirror to this procedure, we should use the affine cover of the toric variety coming from the maximal cones, plus locally-free resolutions of the complexes 
mirror to the Lagrangians.  Hence, we obtain a dg-algebra from the corresponding \u{C}ech resolution.
 However, due to the fact that much of this technology is not currently available, this is not the approach we will take in Section \ref{subsec: example}.
\label{rem: Cech}
\end{remark}

%
%
%
%

\subsection{An example} \label{subsec: example}

To illustrate the program outlined above, we examine the case of a blow up of $\mathbb{P}^2$ in three
infinitesimally close points. Precisely,  take $\P^2$, blow it up at a point, and then blow up at the two points on the exceptional locus corresponding to two coordinate axes passing through the original point.  This is toric.  Namely, we take $X$ to be the toric variety defined by the fan $\Sigma$ in $N_\mathbb{R}$ with one cone generators 
\begin{equation} \Sigma (1) = \{(-1,-1) (1,0), (0,1), (1,1), (1,2), (2,1)\}. \end{equation}
The fan $\Sigma$ is pictured in Figure~\ref{fig: fan} with rays labeled according to the ordering in the above equation.  We name the corresponding toric divisors $\{D_0,\ldots D_5\}$.  

\begin{figure}[h] 
  \begin{tikzpicture}
  [scale=.57, vertex/.style={circle,draw=black!100,fill=black!100,thick, inner sep=0.5pt,minimum size=0.5mm}, cone/.style={->,very thick,>=stealth}]
  \filldraw[fill=black!20!white,draw=white!100]
    (-4.2,4.2) -- (4.2,4.2) -- (4.2,-4.2) -- (-4.2,-4.2) -- (-4.2,4.2);
  \draw[cone] (0,0) -- (5.4,0);
  \draw[cone] (0,0) -- (0,5.4);
  \draw[cone] (0,0) -- (5.1,2.55);
  \draw[cone] (0,0) -- (4.5,4.5);
  \draw[cone] (0,0) -- (-4.5,-4.5);
  \draw[cone] (0,0) -- (2.55,5.1);
  \node at (0,5.8) {$2$};
  \node at (2.7,5.4) {$4$};
  \node at (4.8,4.8) {$3$};
  \node at (5.8,0) {$1$};
  \node at (-4.8,-4.8) {$0$};
  \node at (5.4,2.7) {$5$};
  \foreach \x in {-4,-3,...,4}
  \foreach \y in {-4,-3,...,4}
  {
    \node[vertex] at (\x,\y) {};
  }
  \end{tikzpicture}
 \caption{\label{fig: fan} The fan of the toric variety $X$}
\end{figure}
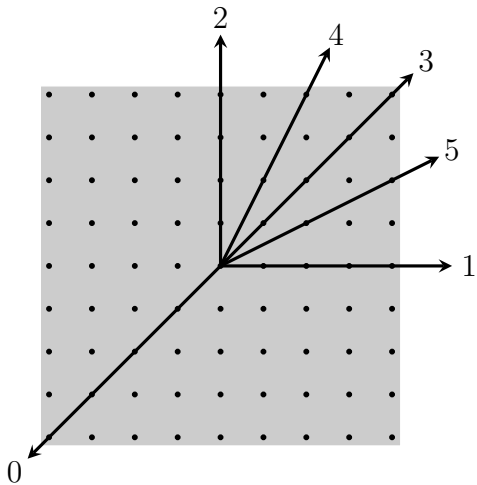

This fan determines the linear map $\gamma$ from $\Z^6$ to $\widehat{G} \cong \Z^4$ displayed in Figure~\ref{fig: linear map} with basis for $\Z^4$ given by $D_0, D_3,D_4,D_5$.

\begin{figure}[h]
\begin{tikzpicture}
 \node (a) at (-3.5,0) {$\Z^6$};
 \node (b) at (3.5,0) {$\Z^4 \cong \widehat{G}$};
 \draw[->] (a) -- node[above] {$\gamma = \begin{pmatrix} 
				   1 & 1 & 1 & 0 & 0 & 0 \\ 
				   0 & -1 & -1 & 1 & 0 & 0 \\
				   0 & -2 & -1 & 0 & 1 & 0 \\
				   0 & -1 & -2 & 0 & 0 & 1 \\
				  \end{pmatrix}$} (b);
\end{tikzpicture}
 \caption{\label{fig: linear map} The map $\gamma$ for the toric variety $X$}
\end{figure}

The variety $X$ is not nef-Fano. However, if
$\Sigma^\prime (1) = \Sigma (1) \backslash \{(1,1)\}$ and  $X^\prime$ is the smooth stack associated to $\Sigma^\prime$,
one observes that $X^\prime$ is nef-Fano and $X^\prime$ is a phase of $X$.  Figure~\ref{fig: minimal model run}  illustrates a straight line
run of the toric Mori program for $X'$ with $\Sigma^\prime$ appearing on the far left and $\Sigma$ appearing as the first phase.  The one-parameter subgroups determining the wall-crossings are listed above each fan modification in the basis for $D_0, D_3, D_4, D_5$.

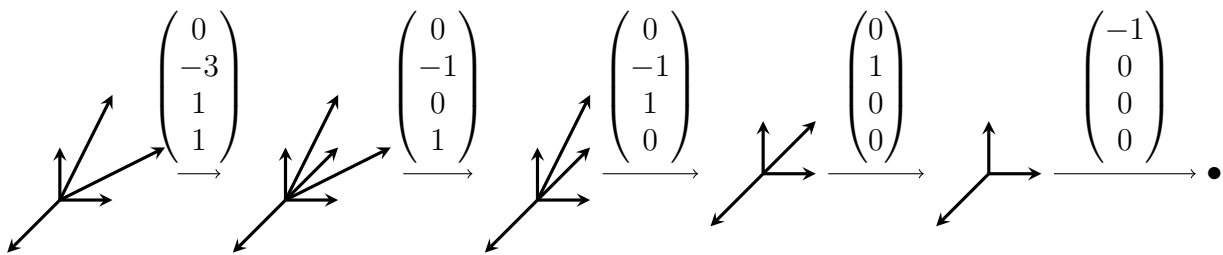
\begin{figure}[h]
\begin{tikzpicture}
 
 \node (a) at (-7.5,0) {
 
  \begin{tikzpicture}
  [scale=.7, vertex/.style={circle,draw=black!100,fill=black!100,thick, inner sep=0.5pt,minimum size=0.5mm}, cone/.style={->,very thick,>=stealth}]
  \draw[cone] (0,0) -- (1,0);
  \draw[cone] (0,0) -- (0,1);
  \draw[cone] (0,0) -- (2,1);
  \draw[cone] (0,0) -- (-1,-1);
  \draw[cone] (0,0) -- (1,2);
  \end{tikzpicture}
 
 };
  
 \node (b) at (-4.5,0) {
  \begin{tikzpicture}
  [scale=.7, vertex/.style={circle,draw=black!100,fill=black!100,thick, inner sep=0.5pt,minimum size=0.5mm}, cone/.style={->,very thick,>=stealth}]
  \draw[cone] (0,0) -- (1,0);
  \draw[cone] (0,0) -- (0,1);
  \draw[cone] (0,0) -- (2,1);
  \draw[cone] (0,0) -- (1,1);
  \draw[cone] (0,0) -- (-1,-1);
  \draw[cone] (0,0) -- (1,2);
  \end{tikzpicture}
 };
 
  \node (c) at (-1.5,0) {
  \begin{tikzpicture}
  [scale=.7, vertex/.style={circle,draw=black!100,fill=black!100,thick, inner sep=0.5pt,minimum size=0.5mm}, cone/.style={->,very thick,>=stealth}]
  \draw[cone] (0,0) -- (1,0);
  \draw[cone] (0,0) -- (0,1);
  \draw[cone] (0,0) -- (1,1);
  \draw[cone] (0,0) -- (-1,-1);
  \draw[cone] (0,0) -- (1,2);
  \end{tikzpicture}
 };
 
  \node (d) at (1.5,0) {
  \begin{tikzpicture}
  [scale=.7, vertex/.style={circle,draw=black!100,fill=black!100,thick, inner sep=0.5pt,minimum size=0.5mm}, cone/.style={->,very thick,>=stealth}]
  \draw[cone] (0,0) -- (1,0);
  \draw[cone] (0,0) -- (0,1);
  \draw[cone] (0,0) -- (1,1);
  \draw[cone] (0,0) -- (-1,-1);
  \end{tikzpicture}
 };
 
  \node (e) at (4.5,0) {
  \begin{tikzpicture}
  [scale=.7, vertex/.style={circle,draw=black!100,fill=black!100,thick, inner sep=0.5pt,minimum size=0.5mm}, cone/.style={->,very thick,>=stealth}]
  \draw[cone] (0,0) -- (1,0);
  \draw[cone] (0,0) -- (0,1);
  \draw[cone] (0,0) -- (-1,-1);
  \end{tikzpicture}
 };
 
  \node (f) at (7.5,0) {
  $\bullet$
 };
 
 \draw[->] (a) -- node[above] {$\begin{pmatrix} 0 \\ -3 \\ 1 \\ 1 \end{pmatrix}$} (b);
 \draw[->] (b) -- node[above] {$\begin{pmatrix} 0 \\ -1 \\ 0 \\ 1 \end{pmatrix}$} (c);
 \draw[->] (c) -- node[above] {$\begin{pmatrix} 0 \\ -1 \\ 1 \\ 0 \end{pmatrix}$} (d);
 \draw[->] (d) -- node[above] {$\begin{pmatrix} 0 \\ 1 \\ 0 \\ 0 \end{pmatrix}$} (e);
 \draw[->] (e) -- node[above] {$\begin{pmatrix} -1 \\ 0 \\ 0 \\ 0 \end{pmatrix}$} (f);
 
\end{tikzpicture}
 \caption{\label{fig: minimal model run} A straight line
run of the toric Mori program for $X^\prime$ marked by one-parameter subgroups}
\end{figure}


Following the outline given in Section \ref{sec: A side}, the straight line
run of the toric Mori program appearing in Figure~\ref{fig: minimal model run} can also be interpreted in the language of monotone path polytopes.  To this end, consider the secondary polytope of $A = \Sigma (1) \cup
\{(0,0)\}$. Recall that the normal fan of the polytope extends the GIT fan of $X$. A direct calculation finds that the secondary polytope $\Sigma (A)$ is a $4$-dimensional polytope with $30$ vertices, $20$ of which have normal cones in the GIT fan of $X$. The monotone path polytope of $\Sigma (A)$ with respect to the function $\chi$ can also be computed as a $3$-dimensional polytope with $24$ vertices. In this case, every
run of the toric Mori program for $X^\prime$ is a straight line run and so the vertices of the monotone path polytope are in one to one correspondence with straight line
runs of the toric Mori program of $X^\prime$. Of these runs, precisely $10$ have $X$ as a phase. The  monotone path polytope is illustrated in Figure~\ref{fig: monotone path polytope}. Those faces which have vertices whose
runs of the toric Mori program contain $X$ as a phase are colored in green. The left most green vertex corresponds to the straight line
run of the toric Mori program illustrated in Figure~\ref{fig: minimal model run}. 

\begin{figure}[h]

\includegraphics{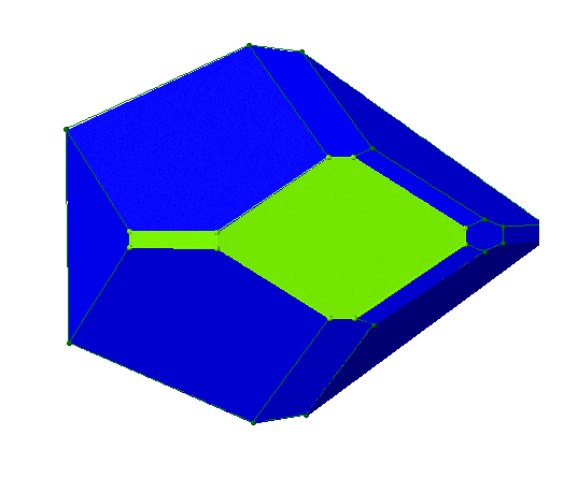}
\caption{\label{fig: monotone path polytope} The monotone path polytope $\Sigma_\chi (\Sigma (A))$.}
\end{figure}

\subsubsection{The $B$-model computation}
Let us now discuss the dg-algebra associated to this straight line
run of the toric Mori program on the $B$-side.  

Recall that $X$ is, by definition, the toric variety defined by the fan $\Sigma$ in
$N_\mathbb{R}$ with one cone generators labeled from $0$ to $5$ by 
\begin{equation} \Sigma (1) = \{(1,0), (2,1), (1,1), (1,2), (0,1), (-1,-1)\}. \end{equation}

The Cox construction describes $X$ as a GIT quotient of the prime spectrum of the homogeneous coordinate ring, $\C[x_0, \ldots,
x_5]$, by the torus dual to the Picard group of $X$, 
\[
G := \op{Hom}(\op{Pic}(X), \mathbb{G}_m),
\]
for any choice of a $G$-equivariant ample line bundle.   In this case,
the semi-stable locus and stable locus agree.  They are the complement of the vanishing locus of the ideal
\[
J := \langle x_3x_4x_2x_0, x_4x_2x_0x_1, x_2x_0x_1x_5, x_0x_1x_5x_3, x_1x_5x_3x_4, x_5x_3x_4x_2 \rangle.
\]  
Let
\[
U := \op{Spec }\C[x_0, \ldots, x_5] \backslash V(J)
\]
so that
\begin{equation} \label{eq: X as GIT}
X \cong [U / G].
\end{equation}

Via the above isomorphism, one can identify $D_i$ with the vanishing locus of $x_i$  for $0 \leq i \leq5$.
With this interpretation, the two disjoint exceptional loci from the second blow-up are $D_4$ and $D_5$
 while the strict transform of the exceptional locus from the first blow-up is described by $D_3$.  The remaining toric divisors, $D_0, D_1$ and $D_2$ are all linearly equivalent to the pullback of a generic hyperplane from $\P^2$.  

Now, the run of the Mori program appearing in Figure~\ref{fig: minimal model run} (starting from the ample chamber of $X$ and ending outside the support of the GKZ fan) corresponds to the sequence of blow-downs for $X$ given by $D_2$ followed by $D_4$ followed by $D_3$.
Consider the blow down maps
\[
f_5: X \to Y
\]
contracting $D_5$,
\[
f_4: Y \to \mathbb{F}_1
\]
contracting $D_4$, and
\[
g: \mathbb{F}_1 \to \P^2
\]
contracting $D_3$.

Table~\ref{tab: weights} displays the weights of each of the one-parameter subgroups appearing in Figure~\ref{fig: minimal model run} along with the data $(W,\mu)$ used to determine the semi-orthogonal decomposition in Theorem~\ref{thm: SOD on B-side}.  The one-parameter subgroup $\lambda_1$ determines the variation of GIT quotients from $X'$ to $X$.  Therefore, to compute the semi-orthogonal decomposition of $\dbcoh{X}$ we consider the sequence of wall-crossings corresponding to $\lambda_2, \lambda_3, \lambda_4, \lambda_5$.

\begin{table}[h]
\begin{center}
 \begin{tikzpicture}
  \matrix (m) [matrix of math nodes, row sep=1em, column sep=1em]
    {  & x_0 & x_1  & x_2  & x_3 & x_4 & x_5 & W_i & \mu_i \\
      \lambda_1 & 0 & 0 & 0 & -3 & 1 & 1 & pt & -1 \\ 
      \lambda_2 & 0 & -1 & 0 & -1 & 0 & 1 & pt & -1 \\ 
      \lambda_3 & 0 & 0 & -1 & -1 & 1 & 0 & pt & -1 \\
      \lambda_4 & 0 & -1 & -1 & 1 & 0 & 0 & pt & -1 \\ 
      \lambda_5 & -1 & -1 & -1 & 0 & 0 & 0 & pt & -3 \\ 
    };
   \draw (5,2.1) -- (-4.1,2.1);
   \draw (-4.1,-3) -- (-4.1,2.1);
 \end{tikzpicture}
\end{center}
 \caption{\label{tab: weights} The weights and wall data associated to each one-parameter subgroup}
\end{table}

Let us first justify the computation of the wall data $(W_2, \mu_2)$ in the case of the wall $\tau_2$ determined by $\lambda_2$.  The anticanonical divisor on $X$ is given by the sum of the $D_i$. In the basis $D_0, D_3, D_4, D_5$ we can sum the columns of $\gamma$ in Figure~\ref{fig: linear map} to obtain
\[
-K_X = -3D_0 +D_3+2D_4+2D_5.
\]
Therefore we can take the dot product of $\lambda_2 = (0,-1,0,1)$ with the anticanonical $(3,-1,-2,-2)$ to obtain $\mu_2 = -1$.  In the same manner, one can calculate each of the $\mu_i$ 

Now recall that $W_2$ is defined as the GIT quotient of the fixed locus of $\lambda_2$ by $(\mathbb{G}_m)^3$ determined by a character $\alpha$ in the relative interior of $\tau_2$.  From Table~\ref{tab: weights}, we see that the fixed locus of $\lambda_2$ is defined by the $x_i$ of weight zero, namely $x_0,x_2,x_4$.  This yields a GIT quotient of $\C^3$ by $(\mathbb{G}_m)^3$. Under change of basis, the action can be diagonalized completely so that weights of the action are $(1,0,0),(0,1,0),$ and $(0,0,1)$. Any GIT quotient with a generic character yields a point. 
Also in the same manner, one can see that each $W_i$ is a point.

Let us consider the application of Theorem~\ref{thm: SOD on B-side} to the wall-crossing associated to $\lambda_2$ with wall data $(pt, -1)$.  For each $d \in \Z$, we get
 fully-faithful functors
\[
\Phi_d: \dbcoh{Y} \to \dbcoh{X}
\]
and
\[
\Upsilon_+ : \dbcoh{pt} \to \dbcoh{X}
\]
and a semi-orthogonal decomposition
\[
\dbcoh{X} = \langle \dbcoh{pt}, \dbcoh{Y} \rangle.
\]

Fix $d =-1$.  The functor $\Upsilon_+$ can be described as follows. The fixed locus of $\lambda_2$ is 
\begin{displaymath}
 \mathbb{A}^3 = V(x_0,x_1,x_2) \subset \mathbb{A}^6. 
\end{displaymath}
The set of points,
\begin{displaymath}
 \{x \in \mathbb{A}^6 \mid \lim_{t \to 0} \lambda_2(t) \cdot x \text{ exists} \} \cong \mathbb{A}^5,
\end{displaymath}
is the vanishing locus of $x_5$. (Here, it useful to remember that the passing from the rings to spectra inverts the weights of an action.)

There is a sequence of morphisms of stacks,
\begin{displaymath}
 \mathbb{A}^3 \modmod{} \ \mathbb{G}_m^3 \to \mathbb{A}^5 \modmod{} \ \mathbb{G}_m^4 \to \mathbb{A}^6 \modmod{} \ \mathbb{G}_m^4 = X. 
\end{displaymath}
The first map is pulling back by the map, $$x \mapsto \lim \lambda_2(t) \cdot x,$$ and using the projection, $$\mathbb{G}_m^4 \cong G \to G/\lambda_2(\mathbb{G}_m) \cong \mathbb{G}_m^3.$$ The second is the inclusion. 

Applying Theorem \ref{theorem: SOD of derived categories for toric variations}, we get a semi-orthogonal decomposition
\[
\dbcoh{X} = \langle \O_{D_5}(-K_X), f_5^*\dbcoh{Y} \rangle.
\]

Proceeding similarly for the walls $\tau_3,\tau_4,\tau_5$ corresponding to the one-parameter subgroups $\lambda_3,\lambda_4,\lambda_5$, the exceptional collection of $\dbcoh{X}$ corresponding to the run of the Mori program in Figure~\ref{fig: minimal model run} is given by
\[
\langle \O_{D_4}(-K_X), \O_{D_5}(-K_X),  (f_4 \circ f_5)^*(\O_{D_3}(-3D_0+D_3)), (g \circ f_4 \circ f_5)^*\dbcoh{\P^2} \rangle.
\]

\begin{remark}
In this way, one can also obtain a semi-orthogonal decomposition of $\dbcoh{X'}$,
\[
\langle \O_p(-K_{X'}), \Phi_d(\dbcoh{X}) \rangle
\]
where $p$ is a stacky point defined by the intersection of $D_4$ and $D_5$ in $X'$.
\end{remark}

It is natural for computational reasons to move the wall contributions from the left to the right. This amounts to applications of Serre functors for some of the admissible subcategories. This leaves us with
\[
\dbcoh{X} = \langle (g \circ f_4 \circ f_5)^*\dbcoh{\P^2}, (f_4 \circ f_5)^*\O_{D_3}, \O_{D_4}, \O_{D_5} \rangle.
\]

\begin{remark}
 The reader may be concerned that we mutated our exceptional before comparing it with the $A$-side. The difference between this choice and the original collection on the $A$-side amounts to a choice of orientation for the vanishing paths. It is therefore inconsequential. 
\end{remark}

Now, notice that 
\[
(f_4 \circ f_5)^*\O_{D_3} = \O_{D_3+D_4+D_5}.
\]
Furthermore, the unique run of the Mori program for the toric GIT description of $\P^2$ yields a semi-orthogonal decomposition
\[
\dbcoh{\P^2} = \langle \O(-2), \O(-1), \O \rangle.
\]
or
\[
(g \circ f_4 \circ f_5)^*\dbcoh{\P^2} = \langle \O(-2D_0), \O(-D_0), \O \rangle.
\]

In summary we have an exceptional collection $\langle E_0, \ldots, E_5 \rangle$ on $\dbcoh{X}$ where 
\begin{align}
 E_0 &= \mathcal{O} (-2 D_0) , \nonumber \\
 E_1 &= \mathcal{O} (- D_0) , \nonumber \\
 E_2 &= \mathcal{O} ,\nonumber \\ 
 E_3 &= \mathcal{O}_{ D_3 + D_4 + D_5},\nonumber \\
 E_4 &= \mathcal{O}_{D_4},\nonumber \\
 E_5 &= \mathcal{O}_{D_5}. \nonumber
\end{align}

Our aim is to compute the $A_\infty$-algebra 
\begin{equation}
 A = \textnormal{End}^\bullet \left( \oplus_{i = 0}^5 E_i \right)
\end{equation}
As a first step, we compute the underlying cohomology algebra which gives the quiver with relations in Figure~\ref{fig: cohomology quiver}.
\begin{figure}[ht]
\begin{center}
\begin{tikzpicture}[scale=1,level/.style={->,>=stealth,thick}]
	\node at (-4,0) {$\bullet$};
	\node at (-2,0) {$\bullet$};
	\node at (0,0) {$\bullet$};
	\node at (2,0) {$\bullet$};
	\node at (4,1) {$\bullet$};
	\node at (4,-1) {$\bullet$};
	
	\draw[level] (-3.85,.15) .. controls (-3,1) and (-3,1) .. node[above] {$a_0$} (-2.15,.15);
	\draw[level] (-3.8,0) -- node[above] {$a_1$} (-2.2,0);
	\draw[level] (-3.85,-.15) .. controls (-3,-1) and (-3,-1) .. node[above] {$a_2$} (-2.15,-.15);

	\draw[level] (-1.85,.15) .. controls (-1,1) and (-1,1) .. node[above] {$b_0$} (-0.15,.15);
	\draw[level] (-1.8,0) -- node[above] {$b_1$} (-0.2,0);
	\draw[level] (-1.85,-.15) .. controls (-1,-1) and (-1,-1) .. node[above] {$b_2$} (-0.15,-.15);
	
	\draw[level] (.2,0) -- node[above] {$t$} (1.8,0);
	
	\draw[level] (2.2,.1) -- node[above] {$c_4$} (3.8,.9);
	\draw[level] (2.1,.2) .. controls (3,1.5) and (3,1.5) .. node[above] {$d_4$} (3.8,1.1);

	\draw[level] (2.2,-.1) -- node[above] {$c_5$} (3.8,-.9);
	\draw[level] (2.1,-.2) .. controls (3,-1.5) and (3,-1.5) .. node[above] {$d_5$} (3.8,-1.1);
\end{tikzpicture}
\end{center}
\caption{Scissors quiver}
\label{fig: cohomology quiver}
\end{figure}
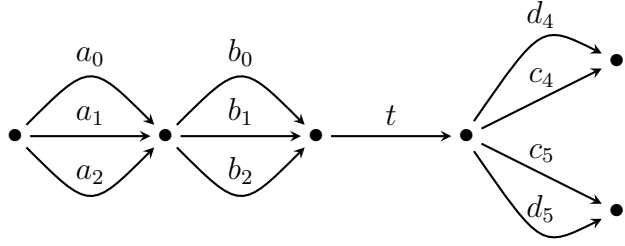


In Figure~\ref{fig: cohomology quiver} the degree of each arrow is zero except 
\begin{equation}
 \deg (d_4) = 1 = \deg (d_5).
\end{equation}

The relations are 
\begin{align}
 b_2 a_1 = b_1 a_2, \nonumber \\
 b_2 a_0 = b_0 a_2, \nonumber \\
 b_1 a_0 = b_0 a_1, \label{eq:relations1} \\
 t b_1 = 0 = t b_2 ,  \nonumber \\
 d_4 t = 0 = d_5 t . \nonumber 
\end{align}
As a formal consequence of the homological grading on the $A_\infty$ higher products, it is possible to eliminate all such higher products except $m^A_3$. 
%
To compute $m_3$, we must resolve the sheaves by either injective sheaves, or by the total complex of a vector bundle and \u{C}ech resolution (see also Remark~\ref{rem: Cech}). However, in this case we may perform one left mutation to obtain another exceptional collection 
\begin{align}
 \langle \tilde{E}_0 , \tilde{E}_1 , \tilde{E}_2 , \tilde{E}_3, \tilde{E}_4, \tilde{E}_5 \rangle &= \langle E_0, E_1, L_{E_2} E_3, E_2, E_4, E_5 \rangle
\end{align}
where the twist can be written concretely as $L_{E_2} E_3 = \mathcal{O} (-D_1 - D_2 - D_3 )$. One can check that this exceptional collection is strong so its endomorphim algebra is formal.

To get a controlled dg-resolution of $A$, we may replace $\mathcal O_{D_3+D_4+D_5}$ by the complex,
\begin{displaymath}
 \mathcal O(-D_3-D_4-D_5) \overset{x_3x_4x_5}{\longrightarrow} \mathcal O,
\end{displaymath}
and consider the resulting dg-quiver in Figure \ref{fig: DG-quiver}.

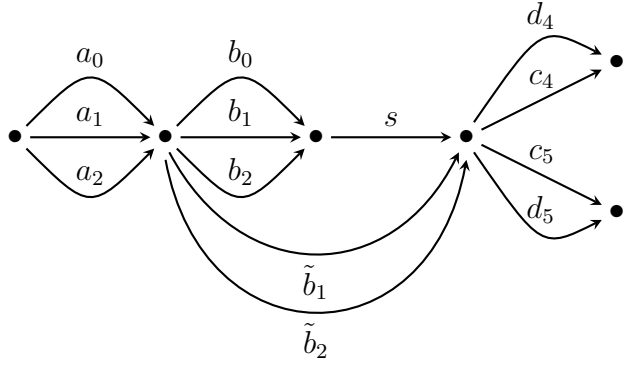
\begin{figure}[ht]
\begin{center}
\begin{tikzpicture}[scale=1,level/.style={->,>=stealth,thick}]
	\node at (-4,0) {$\bullet$};
	\node at (-2,0) {$\bullet$};
	\node at (0,0) {$\bullet$};
	\node at (2,0) {$\bullet$};
	\node at (4,1) {$\bullet$};
	\node at (4,-1) {$\bullet$};
	
	\draw[level] (-3.85,.15) .. controls (-3,1) and (-3,1) .. node[above] {$a_0$} (-2.15,.15);
	\draw[level] (-3.8,0) -- node[above] {$a_1$} (-2.2,0);
	\draw[level] (-3.85,-.15) .. controls (-3,-1) and (-3,-1) .. node[above] {$a_2$} (-2.15,-.15);

	\draw[level] (-1.85,.15) .. controls (-1,1) and (-1,1) .. node[above] {$b_0$} (-0.15,.15);
	\draw[level] (-1.8,0) -- node[above] {$b_1$} (-0.2,0);
	\draw[level] (-1.85,-.15) .. controls (-1,-1) and (-1,-1) .. node[above] {$b_2$} (-0.15,-.15);
	
	\draw[level] (-1.95,-.2) .. controls (-1,-2) and (1,-2) .. node[below] {$\tilde{b}_1$} (1.9,-.2);
	\draw[level] (-2,-.3) .. controls (-1.5,-3) and (1.5,-3) .. node[below] {$\tilde{b}_2$} (2,-.3);
	
	\draw[level] (.2,0) -- node[above] {$s$} (1.8,0);
	
	\draw[level] (2.2,.1) -- node[above] {$c_4$} (3.8,.9);
	\draw[level] (2.1,.2) .. controls (3,1.5) and (3,1.5) .. node[above] {$d_4$} (3.8,1.1);

	\draw[level] (2.2,-.1) -- node[above] {$c_5$} (3.8,-.9);
	\draw[level] (2.1,-.2) .. controls (3,-1.5) and (3,-1.5) .. node[above] {$d_5$} (3.8,-1.1);
\end{tikzpicture}
\end{center}
\caption{Resolved scissors quiver}
\label{fig: DG-quiver}
\end{figure}

All arrows are degree zero except
\begin{gather}
 \op{deg}( d_4 ) = \op{deg} (d_5) = 1 \\
 \op{deg}(\tilde{b_1}) = \op{deg} ( \tilde{b_2} ) = -1.
\end{gather}

The relations on arrows in the quiver illustrated in Figure \ref{fig: DG-quiver} are 

\begin{gather}
 b_ja_i = b_i a_j \ 0 \leq i,j \leq 2 \nonumber \\
 \tilde{b}_1 a_2 = \tilde{b}_2 a_1 \nonumber \\
 d_j\tilde{b}_i = 0 \ 1 \leq i \leq 2 \ , \ 4 \leq j \leq 5 \nonumber \\
 c_jsb_i = 0 \ 0 \leq i \leq 3 \ , \ 4 \leq j \leq 5  \\
 d_jsb_i = 0 \ 1 \leq i \leq 2 \ , \ 4 \leq j \leq 5 \nonumber \\
 c_4 \tilde{b}_2 = d_4 s b_0 \ , \ c_5 \tilde{b}_1 = d_5 s b_0 \nonumber \\
 c_4 \tilde{b}_1 = c_5 \tilde{b}_2 = 0 \nonumber
\end{gather}

The differential is zero on all arrows except 
\begin{gather}
 d( \tilde{b}_i ) = sb_i, 1 \leq i \leq 2.
\end{gather}
The differential then extends linearly over the path algebra. 

To obtain the higher morphisms on $A$, one implements homological perturbation (see Section~\ref{sec: perturbation})
\begin{center}
 \begin{tikzpicture}
  \node (a) at (-1,0) {$\widehat{A}$};
  \node (b) at (1,-.075) {$A$};
  \draw[->] (a) .. controls (-2,1) and (-2,-1) .. node[left] {$h$} (a);
  \draw[->] (-.7,.1) -- node[above] {$\pi$} (.7,.1);
  \draw[->] (.7,-.2) -- node[below] {$i$} (-.7,-.2);
 \end{tikzpicture}
\end{center}
where $i$ is the inclusion using the already enforced notational identification and setting $i(t) = s$. The map $\pi$ is the obvious projection and $h$ is the homotopy which vanishes except for 
\begin{gather}
 h(s b_i )  =  \tilde{b}_i \ , \ 1 \leq i \leq 2  \\
 h(s b_i a_j) = h(s b_j a_i) = \tilde{b}_i a_j \ , \ 1 \leq i \leq 2, 0 \leq j \leq 3.
\end{gather}
To calculate $m^A_3$ we use the formula \eqref{eq: m3} obtained by homological perturbation in Section~\ref{sec: perturbation}  which we recall here.
\begin{equation}
 m^A_3 (x ,y , z) =\pi (m_{\tilde{A}}^2 ( h ( m^A_2 (i(x), i(y)) ) , i(z)) ) - (-1)^{|x|} m_{\tilde{A}}^2 (i(x),  h ( m^A_2 (i(y), i(z)) ) ) ) . 
\end{equation}
This can be computed to give zero on all triples other than three possible sequences of elements in $A$. We obtain  
\begin{align}
 m_3^A (d_4, t , b_2 ) & = c_4 t b_0 , & m_3^A (d_5 , t , b_1 ) &= c_5 t b_0 , \\
 m_3^A (d_4 , t , b_2 a_0  ) &= c_4 t b_0 a_0 , & m_3^A (d_5 , t ,  b_1 a_0) &= c_5 t b_0 a_0\\
 m_3^A (d_4 , t b_0 , a_2 ) &= c_4 t b_0 a_0 , & m_3^A (d_5 , t b_0 , a_1) &= c_5 t b_0 a_0
\end{align}

\subsubsection{The $A$-model computation}

The mirror picture of the run of the toric Mori program appearing in Figure~\ref{fig: minimal model run} follows from the monotone path polytope description.  It is a degenerated LG model $\{\psi (t)\}$. The cycle corresponding to the maximal degeneration
$\psi (0)$ has a sequence of components $\{Z_1, Z_2, Z_3, Z_4, Z_5\}$ which are weighted projective
line orbits in $\mathcal{X}_{\Sigma (A)}$. Combinatorially, this corresponds to a path in $\Sigma (A)$ with vertices
$P = \{T_0, \ldots, T_5\}$ associated to triangulations of $(Q,A)$ and edges $\{T_i, T_{i + 1}\}$ associated to extended circuits.  Note that this is just a reinterpretation of the choice of a run of the toric Mori program. 

Utilizing the formula given in \eqref{eq:mult form}, the ramifications of the components $\psi_i (0)$ of $\psi (0)$ can be
computed using the vertices $\phi_{T_i}$ to be $\{3,1,1,1,1\}$. Indeed, this formula amounts to finding the difference between the sum of the volumes of all triangles containing $0$ as a vertex in $T_{i + 1}$ with the same sum for $T_i$. Inspecting Figure~\ref{fig:tropical Morse mirror}, one sees that no triangles contain $0$ in $T_0$, while three contain it $T_1$ yielding $\left< e_0 ,\varphi_{T_1} - \varphi_{T_0}  \right> = 3$. For each remaining edge, a single unit volume triangle is added to the set containing zero as a vertex, so each has multiplicity $a_i = 1$.  As there are no degenerate circuits in this sequence, this also gives the number of critical values in each annular decomposition of $\psi (t)$ for $t \approx 0$.

\begin{figure}[h]
\includegraphics{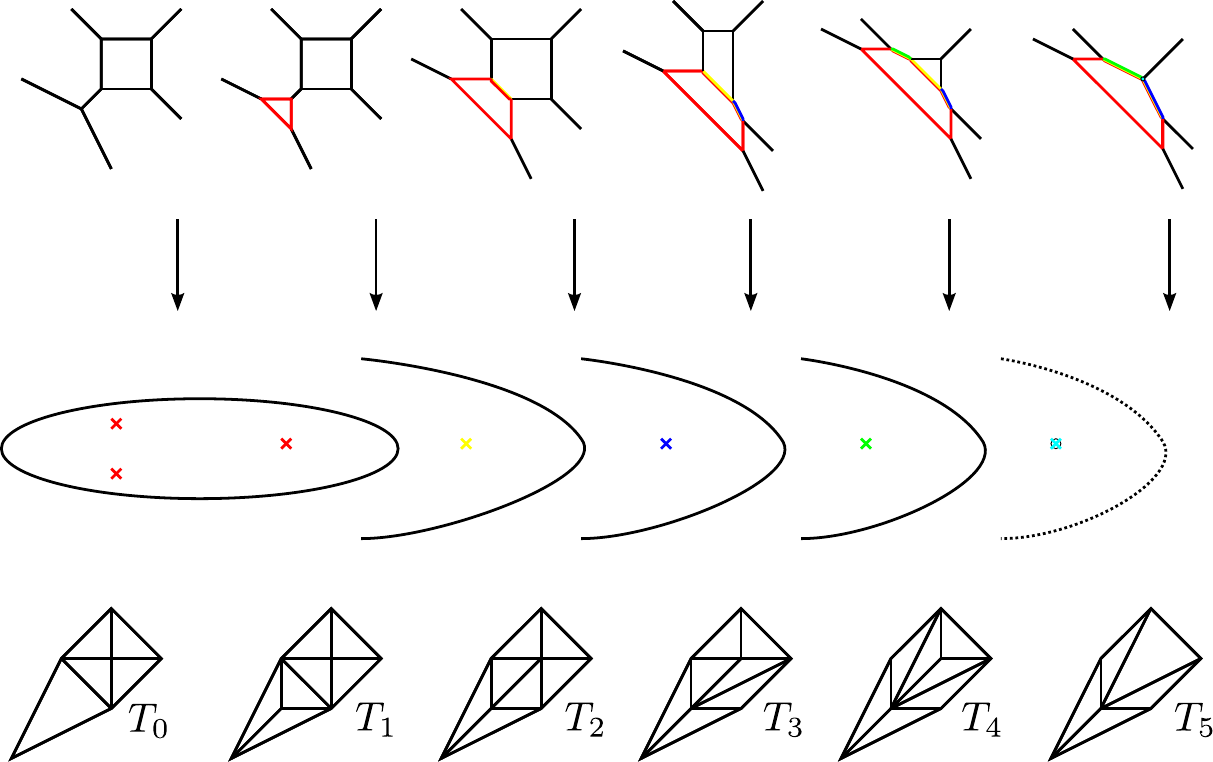}
\caption{\label{fig:tropical Morse mirror} The potential near the mirror degeneration.}
\end{figure}

Figure~\ref{fig:tropical Morse mirror} illustrates, the triangulations $T_i$, together with the critical fibers arranged into annuli as $t$ approaches zero and the tropical limit of the hypersurface over $\psi(t)$ near the bounding radii of each annulus.  Each critical fiber is color-coordinated with the corresponding critical value.

As a sanity check with the $B$-side, let us pause for a moment to consider Figure~\ref{fig:tropical fiber} which displays the tropicalization of the fiber corresponding to the outer radius of the annulus
associated to $Z_4$ (or equivalently the tropical variety dual to $T_5$).   
Notice that the two vanishing cycles (pictured in red and blue) are disjoint.  Applying Proposition~\ref{prop:tropical vanishing}, one confirms immediately that the semi-orthogonal components
$\fs{\psi_3 (t)}$ and $\fs{\psi_4 (t)}$ are in fact orthogonal. This is certainly the case on the $B$-side, as these two
categories are generated by structure sheaves of disjoint exceptional loci.

\begin{figure}[h]
 \includegraphics{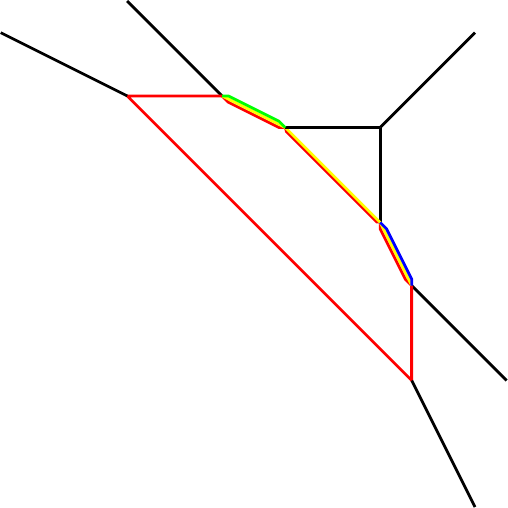}
 \caption{\label{fig:tropical fiber} The mirror fiber and tropical vanishing cycles for $X$.}
\end{figure}

To each critical fiber in Figure~\ref{fig:tropical Morse mirror} we also associate a color-coordinated vanishing cycle in Figure~\ref{fig:fiber}. This would be done, theoretically, using the prescription in Section~\ref{subsubsec: tropical} which characterizes the vanishing cycles associated to a circuit as stable tropical Morse cells.  However, 
as the calculus mentioned in Section \ref{sec: general method} has not yet been developed, we instead utilize the local model for vanishing cycles of surface circuits obtained in Figure~\ref{fig:circuit vanishing cycles}. 

Near the first component $Z_0$ in the degeneration, $\psi(t)$ has $3$ critical values whose vanishing cycle consists of the $(1,3)$ circuit vanishing cycle $L_2$ drawn in Figure~\ref{fig:circuit vanishing cycles} along with two vanishing cycles.
 The first of these, $L_1$, is obtained by Dehn twisting once around each cylinder corresponding to an internal edge of the stable complex and the second, $L_0$, is obtained by Dehn twisting twice along these cylinders. 
For each of the remaining critical values, we obtain the local model of a $(2,2)$ circuit.
 As the vanishing cycle flows along its distinguished basis, one must keep careful track of its image in the fiber.
  Assembling these models onto the fiber near the outer radii of the $Z_4$ degeneration produces Figure~\ref{fig:fiber}. 

\begin{figure}[h]
 \includegraphics{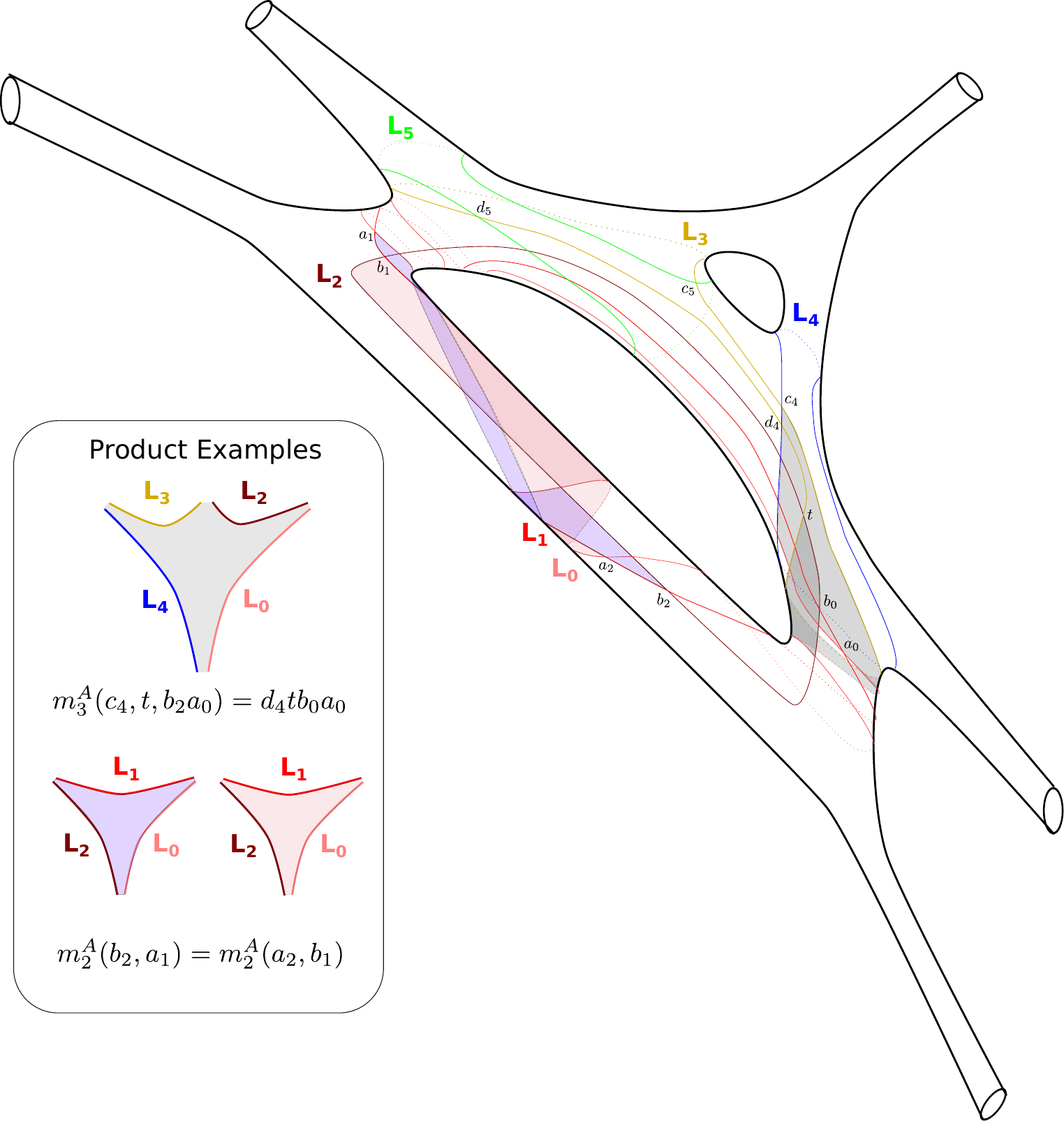}
 \caption{\label{fig:fiber} The fiber $F$ with vanishing cycles and one example of a higher product}
\end{figure}

We conclude with a visual inspection of the $A_\infty$-algebra associated to the cycles $\oplus_{i = 0}^5 L_i$ in $\fs{\psi}$ as illustrated in  Figure~\ref{fig:fiber}. 
To this end, we have labeled the intersections of the vanishing cycles in bijection with the morphisms obtained in the previous section for the $B$-model. 

 For the rigorous definition of the category associated to a collection of vanishing cycles in the general setting  of symplectic Lefschetz pencils, see \cite{SeiBook}.  We briefly recall the definition of the higher products $m_k$ now .
 
 The $A_\infty$-structure maps $\{m_k \}_{k  \geq 1}$ are determined by a count of holomorphic polygons with appropriate boundary conditions. More precisely, let $L_{i_k}, \ldots, L_{i_1}$ be $k$ distinct vanishing cycles in the collection, with $i_1 < \cdots < i_k$. Here the subscript indicates to which path in the distinguished basis the vanishing cycle corresponds. By a $(k + 1)$-marked disc $\Delta$, we mean the disc in $\mathbb{C}$ with $(k + 1)$ marked points $\{p_0, p_1, \ldots, p_k \}$ on the boundary oriented counterclockwise. We think of the $k$ points $p_i$ as incoming points and $p_0$ as an outgoing point and write $\partial_i \Delta$ for the arc on the boundary of $\Delta$ between $p_i$ and $p_{i + 1}$.
 
 Finally, take $\mathcal{M}_k (F , L_{i_k}, \ldots L_{i_0})$ to be the moduli space of holomorphic maps $u$ from a $(k + 1)$-marked disc $\Delta$ to $F$ such that $u (\partial_j \Delta ) \subset L_{i_j}$. Let $\mathcal{M}^0_k ( F , L_{i_k}, \ldots L_{i_0})$ be the zero dimensional subspace in $\mathcal{M}_k  (F , L_{i_k}, \ldots L_{i_0})$. In general, this space is determined by considering maps with Maslov index $2$, which in the $1$ dimensional setting amounts to restricting to holomorphic immersions. For exact manifolds like $F$, the zero dimensional moduli spaces are compact and can be given an orientation, so that we obtain an integer $\chi ( \mathcal{M}_k  (F , L_{i_k}, \ldots L_{i_0}) ) \in \mathbb{Z}$ equal to the Euler characteristic. Then for $p_{i_j} \in L_{i_j} \cap L_{i_{j + 1}} = \textnormal{Hom}^* (L_{i_j}, L_{i_{j + 1}} )$ we compute the $k$-th $A_\infty$-product as 
\begin{equation} m_k (a_{i_k} , \ldots, p_{i_0}) = \sum_{p_0 \in L_{i_k} \cap L_{i_0} } \chi ( \mathcal{M}_k  (F , L_{i_k}, \ldots L_{i_0}) ) p_0 . \label{eq: product} \end{equation}

Applying this formula to the vanishing cycles $\{L_0, \ldots, L_5 \}$ in the fiber $X$, one sees the underlying scissors quiver in Figure~\ref{fig: cohomology quiver} based on the marking of arrows in Figure~\ref{fig:tropical Morse mirror}.  To verify the relations obtained for the $B$-model, one must count holomorphic triangles according to \eqref{eq: product}.  One such count is made between the Lagrangians $L_0, L_1, L_2$ in Figure~\ref{fig:tropical Morse mirror}.  It produces the relation
\[
m^A_2(b_2, a_1) = m^A_2(a_2, b_1).
\]
\sidenote{{\color{red} Signs clear from the picture?} Signs and sign conventions may not be fixable by 1.}
The rest of the relations are handled similarly.

To calculate the higher products $m_3$ one verifies that there are six holomorphic quadralaterals with cyclic boundary conditions.  One such  relation and one quadralateral appears in Figure~\ref{fig:fiber} between the Langrangians $L_0, L_2, L_3, L_4$. It produces higher product
\[
m^A_3(c_4, t, b_2a_0) = d_4tb_0a_0.
\]
The rest can be found at {\color{blue} \href{http://www.math.miami.edu/~gdkerr/products.pdf}{www.math.miami.edu/~gdkerr/products.pdf}}.  Finally, $m_n =0 $  for $n \geq 4$ as a formal consequence of the underlying cohomology algebra just like on the $B$-side.


\end{document}